\documentclass[11pt]{article}

\usepackage{verbatim}
\usepackage{graphicx}
\usepackage{amsmath,amssymb,amsfonts,euscript,enumerate,latexsym}
\usepackage{amsthm}
\usepackage{color,wrapfig}
\usepackage{pxfonts}
\usepackage{fancyhdr}
\usepackage{mathrsfs}
\usepackage{mathtools}
\usepackage{epsfig,subfigure}
\usepackage{marvosym}
\usepackage{txfonts}
\usepackage{hyperref}
\usepackage{times}
\usepackage{enumerate}

\def\titlerunning#1{\gdef\titrun{#1}}
\makeatletter
\def\author#1{\gdef\autrun{\def\and{\unskip, }#1}\gdef\@author{#1}}
\def\address#1{{\def\and{\\\hspace*{18pt}}\renewcommand{\thefootnote}{}%
		\footnote {#1}}%
	\markboth{\autrun}{\titrun}}
\makeatother
\def\email#1{e-mail: #1}

\def\keywords#1{\par\medskip
	\noindent\textbf{Keywords.} #1}

\usepackage{etoolbox}
\makeatletter
\patchcmd{\@maketitle}
{\ifx\@empty\@dedicatory}
{\ifx\@empty\@date \else {\vskip3ex \centering\footnotesize\@date\par\vskip1ex}\fi
	\ifx\@empty\@dedicatory}
{}{}
\patchcmd{\@adminfootnotes}
{\ifx\@empty\@date\else \@footnotetext{\@setdate}\fi}
{}{}{}
\makeatother

\newtheorem{thm}{Theorem}[section]
\newtheorem{cor}[thm]{Corollary}
\newtheorem{lemma}[thm]{Lemma}

\newtheorem{remark}[thm]{Remark}

\newcommand{\R}{{\mathbb{R}}}

\newcommand{\N}{{\mathbb{N}}}

\newcommand{\vp}{\varphi}

\newcommand{\La}{\triangle}
\newcommand{\bs}{\backslash}

\begin{document}

\baselineskip=17pt

\titlerunning{}

\title{System of Degenerate Parabolic $p$-Laplacian}

\author{ Sunghoon Kim  
\and Ki-Ahm Lee }

\date{}

\maketitle

\address{
Sunghoon Kim (\Letter) :
Department of Mathematics, The Catholic University of Korea, \\
43 Jibong-ro, Wonmi-gu, Bucheon-si, Gyeonggi-do, 14662, Republic of Korea \\
\email{math.s.kim@catholic.ac.kr}
\and
Ki-Ahm Lee :
Department of Mathematical Sciences, Seoul National University, Gwanak-ro 1, Gwanak-Gu, Seoul, 08826, Republic of Korea \& Korea Institute for Advanced Study, Seoul 02455, Republic of Korea \\
\email{ kiahm@snu.ac.kr }
}

\begin{abstract}
	In this paper, we study the mathematical properties of the solution $\bold{u}=\left(u^1,\cdots,u^k\right)$ to the degenerate parabolic system
	\begin{equation*}
		\bold{u}_t=\nabla\cdot\left(\left|\nabla\bold{u}\right|^{p-2}\nabla \bold{u}\right), \qquad \qquad \left(p>2\right).
	\end{equation*}
	More precisely, we show the uniqueness and existence of solution $\bold{u}$ and investigate a priori $L^{\infty}$ boundedness of the gradient of the solution. Assuming that the solution decays quickly at infinity, we also prove that the component $u^l$, $\left(1\leq l\leq k\right)$, converges to the function $c^l\mathcal{B}$ in space as $t\to\infty$. Here, the function  $\mathcal{B}$ is the fundamental or Barenblatt solution of $p$-Laplacian equation and the constant $c^l$ is determined by the $L^1$-mass of $u^l$. The proof is based on the existence of entropy functional.\\
	\indent As an application of the asymptotic large time behaviour, we establish a Harnack type inequality which makes the size of spatial average being controlled by the value of solution at one point.
\end{abstract}

\vskip 0.2truein

\setcounter{equation}{0}
\setcounter{section}{0}

\section{Introduction and Main Results}\label{section-intro}
\setcounter{equation}{0}
\setcounter{thm}{0}

\indent We consider the solution $\bold{u}=\left(u^1,\cdots,u^k\right)$ to the Cauchy problem for the  $p$-Laplacian type parabolic system
\begin{equation}\label{eq-main-equation-of-system}\tag{SPL}
	\begin{cases}
		\begin{aligned}
			\left(u^l\right)_t&=\nabla\left(\left|\nabla\bold{u}\right|^{p-2}\nabla u^l\right) \qquad \mbox{in $\R^n\times\left(0,\infty\right)$}\\
			u^l(x,0)&=u^l_0(x) \qquad \qquad \qquad \forall x\in\R^n,\quad 1\leq l\leq k
		\end{aligned}
	\end{cases}
\end{equation}
in the range of $p>2$, with initial data $u^l_0$ nonnegative, integrable and compactly supported.\\
\indent This type of system has been studied by many authors because of mathematical interest and wide applications. Local boundedness and local $C^{1,\alpha}$ estimates of solution were studied by DeBenedetto,  \cite{D}. It also arose in geometric application \cite{U}, in quasiregular mappings \cite{I} and in fluid dynamics \cite{K}. In particular, in the mid 1960s, Ladyzhenskaya \cite{La} suggested the system above as a model of non-Newtonian fluids, which is formulated by a series of equations having a stress tensor determined by the symmetric part of the gradient of the velocity:
\begin{equation*}
	\begin{aligned}
		\bold{v}_t=\nabla\cdot \mathcal{A}\left(D\left(\bold{v}\right)\right)+\nabla p=\nabla\left(\bold{v}\otimes \bold{v}\right)-f ,\qquad \nabla \bold{v}=0
	\end{aligned}
\end{equation*}
where $\bold{v}=\left(v^1,v^2,v^3\right)$ denotes the velocity, $D\left(\bold{v}\right)$ is the symmetric part of $\nabla\bold{v}$, $p$ is the pressure and $\mathcal{A}$ is a monotone vector field.\\
\indent When $k=1$, the system \eqref{eq-main-equation-of-system} is usually called the parabolic $p$-Laplacian equation, which is well  known in a series of models of gradient-dependent diffusion equations. Large number of literatures on the mathematical theory of parabolic $p$-Laplacian equation can be found. We refer the readers to the papers \cite{KV}, \cite{LPV} for asymptotic behaviour and papers \cite{D}, \cite{DGV} for regularity theory.\\
\indent Another well-known model equation for nonlinear diffusion, which is deeply related to the $p$-laplacian type equation, is the porous media type equation, shortly PME,
\begin{equation*}
	u_t=\La u^m=m\nabla\cdot\left(u^{m-1}\nabla u\right), \qquad m>1.
\end{equation*}
Many mathematical properties of $p$-Laplacian equation can be treated by the systematic study of the PME. For example, the combination of symmetrization, scaling porperties and time discretization can be useful to get the suitable estimates in the theory of $p$-Laplacian. We also considered a few problems of porous medium type system, which were designed to get suitable techniques on the study of degenerate parabolic $p$-Laplacian system \eqref{eq-main-equation-of-system}. In \cite{KL1}, we investigated the local continuity and asymptotic behaviour of the degenerate parabolic system
\begin{equation*}\label{eq-model-eq-in-porous-medium-equation-total-sum-as-diff-coeff}
	\left(u^l\right)_t=m\nabla\cdot\left(\left(\sum_{l=1}^{k}u^l\right)^{m-1}\nabla u^l\right), \qquad m>1, \,\,1\leq l\leq k,
\end{equation*}
which describes the population densities of $k$-species in a closed system whose diffusion are determined by their total population density. In \cite{KL2}, we studied the mathematical properties of the solution $\bold{u}=\left(u^1,\cdots,u^k\right)$ of the degenerate parabolic system
\begin{equation}\label{eq-model-eq-in-porous-medium-equation-norm-as-diff-coeff}
	\left(u^l\right)_t=m\nabla\cdot\left(\left|\bold{u}\right|^{m-1}\nabla u^l\right), \qquad m>1,\,\,1\leq l\leq k.
\end{equation}
Two different methods were considered for the analysis of asymptotic behaviour of \eqref{eq-model-eq-in-porous-medium-equation-norm-as-diff-coeff} and, as a consequence of asymptotic analysis, we established a Harnack type inequality of solution. We also found 1-directional travelling wave type solutions of \eqref{eq-model-eq-in-porous-medium-equation-norm-as-diff-coeff} in \cite{KL2}.\\
\indent As the first result of this paper, we will give the uniqueness and existence of the solution to the parabolic system \eqref{eq-main-equation-of-system}. Recently, many approaches to the existence of solution rely on obtaining suitable regularities which are available on the data. Following such ideas, we first focus on a priori $L^{\infty}$ boundedness estimates of diffusion coefficients of the parabolic system \eqref{eq-main-equation-of-system} in the proof of existence of solution. The statement of our first result is as follow.

\begin{thm}[\textbf{Existence of solution $\bold{u}$ and Uniform $L^{\infty}$ boundedness of $\left|\nabla\bold{u}\right|$}]\label{thm-Main-boundedness-of-diffusion-coefficients-of-p-Laplacian}
	Let $n\geq 2$ and $p>2$. Suppose that $\bold{u}_0\in L^2\left(\R^n\right)\cap W^{1,p}\left(\R^n\right)$. Then there exists a weak solution $\bold{u}=\left(u^1,\cdots,u^k\right)$ of \eqref{eq-main-equation-of-system}. Moreover, there exist a constant $\textbf{M}>0$ such that
	\begin{equation}\label{eq-uniform-boundeness-of-grad-of-soluiton-bold-u}
		\sup_{x\in\R^n,\,\,t\geq T}\left|\nabla\bold{u}\right|(x,t)\leq \frac{\textbf{M}}{T^{\frac{n}{n(p-2)+2p}}}.
	\end{equation}
\end{thm}

In general, solutions of parabolic systems lose continuously their information from the initial data as the time goes by, and evolve only under the laws determined by the system. Therefore, only the diffusion coefficients and external forcing terms will play important roles on the evolution of solution after the large time. Under this observation, we can expect that each component of solution $\bold{u}$ of parabolic system \eqref{eq-main-equation-of-system} satisfies that
\begin{equation}\label{eq-approximation-expectation-of-u-i-in-a-large-time}
	u^l(x,t)\approx c^lf(x,t) \qquad\qquad  \mbox{for sufficiently large $t$}
\end{equation}
for some constants $c^1$, $\cdots$, $c^k$ and function $f$. Then, by \eqref{eq-main-equation-of-system} and \eqref{eq-approximation-expectation-of-u-i-in-a-large-time} $f$ will satisfy
\begin{equation}\label{eq-equation-for-f-approximation-of-u-i-s-at-sufficineltyl-large=-time}
	f_t\approx \left|\bold{c}\right|^{p-2}\nabla\cdot\left(\left|\nabla f\right|^{p-2}\nabla f\right)\qquad \qquad \mbox{for sufficiently large $t$}
\end{equation}
where $\bold{c}=\left(c^1,\cdots,c^k\right)$. Denote by $\mathcal{B}_M$ the fundamental solution of $p$-Laplacian equation, i.e., the function $\mathcal{B}_M$ is a solution of 
\begin{equation*}
	g_t=\textbf{div}\left(\left|\nabla g\right|^{p-2}\nabla g\right) \qquad \left(p>1\right)
\end{equation*}
in a weak sense in $\R^n\times(0,\infty)$ and satisfies
\begin{equation*}
	\mathcal{B}_M(x,0)=M\delta(x) \qquad \left(M>0\right).
\end{equation*}
Then, it is expressed by the following self-similar form
\begin{equation}\label{expression-of-Barenblatt-profile-of-p-Laplaican}
	\mathcal{B}_M(x,t)=t^{-a_1}\left(C_M-\frac{\left(p-2\right)a_2^{\frac{1}{p-1}}}{p}\left(\frac{\left|x\right|}{t^{a_2}}\right)^{\frac{p}{p-1}}\right)_+^{\frac{p-1}{p-2}}
\end{equation} 
where the  constant $C_{M}$ is determined by $L^1$-mass $M$ and the constants $a_1$ and $a_2$ are given by 
\begin{equation}\label{constants-a-1-a-2-for-scaling-and-mass-conservation}
a_1=\frac{n}{(p-2)n+p} \qquad \mbox{and}\qquad a_2=\frac{1}{(p-2)n+p}.
\end{equation}
Since the convergence between solution of \eqref{eq-equation-for-f-approximation-of-u-i-s-at-sufficineltyl-large=-time} and fundamental profile $\mathcal{B}_M$ is well known by \cite{KV} and \cite{LPV}, we cannot choose but think that
\begin{equation*}
	u^l\approx c^lf\to c^l\mathcal{B}_M \qquad \mbox{as $t\to\infty$}.
\end{equation*}
As the second result of our paper, we studied the asymptotic large time behaviour of component $u^l$, $\left(1\leq j\leq k\right)$, through the method of entropy approach. The statement is as follow. 
\begin{thm}[\textbf{Aysmptotic Large Time Behaviour with Entropy Approach}]\label{thm-convergence-between-absolute-bold-theta-and-fundamental-function-with-decay-rate}
	Let $n\geq 2$, $p>2$ and let  $\bold{u}=\left(u^{\,1},\cdots,u^{\,k}\right)$ be a solution of \eqref{eq-main-equation-of-system} with initial data $\bold{u}_0=\left(u^1_0,\cdots,u^k_0\right)$ such that
	\begin{equation*}
		\int_{\R^n}\left(1+\left|\bold{u}_0\right|^{\frac{p-2}{p-1}}+\left|x\right|^{\frac{p}{p-1}}\right)\left|\bold{u}_0\right|\,dx<\infty.
	\end{equation*}
	Then, there exists a constant $C_i>0$ such that
	\begin{equation}\label{1st-conclusion-of-convergence-between-u-i-and-fundamental-function-with-decay-rate}
		\lim_{t\to\infty}\left\|u^l(\cdot,t)-\frac{M_l}{\left|\bold{M}\right|}B_{\left|\bold{M}\right|}(\cdot,t)\right\|_{L^1\left(\R^n\right)}=0, \qquad 1\leq l\leq k,
	\end{equation}
	where $\left\|u_0^l\right\|_{L^1\left(\R^n\right)}=M_l$ for $i=1,\cdots, k$ and $\bold{M}=\left(M_1,\cdots,M_k\right)$.
\end{thm}
The Harnack type inequality play important roles on the the mathematical theories of parabolic differential equations. It is used to obtain upper bounds of a solution of the generalized PME \cite{DK}. It is also the core of investigating the local behaviour of solution to nonlinear parabolic equation (See \cite{CS} for the local structure of free boundaries and \cite{EG} for the boundary behaviour of solutions). Hence, the existence of the system version of Harnack type inequality will be very helpful in studying a variety of mathematical theories on the system.\\
\indent As a generalization, we will consider a suitable Harnack type inequality for the component $u^l$ of a continuous weak solution $\bold{u}=\left(u^1,\cdots,u^k\right)$ of \eqref{eq-main-equation-of-system} in the last part of our paper. One interest thing in our result is that the size of the spatial average of $u^l$ is controlled by the value of it at one point. The result is stated as follow.

\begin{thm}[\textbf{Harnack Type Inequality}]\label{lem-calculation-about-waiting-time-lower-bound}
	Let $p>2$, $T>0$ and let $\bold{u}=\left(u^1,\cdots,u^k\right)$ be a continuous weak solution of 
	\begin{equation}\label{eq-system-PME-for-waiting-time-1}
		\left(u^l\right)_t=\nabla\left(\left|\nabla\bold{u}\right|^{p-2}\nabla u^l\right) \qquad \mbox{in $\R^n\times\left[0,T\right]$}
	\end{equation}
	with the initial data $u^l_0$, $(1\leq l\leq k)$, non-negative, integrable and compactly supported. Suppose that  there exists an uniform constant $\mu_0>0$ such that
	\begin{equation*}
		\frac{\min_{1\leq j\leq k}\left\{\int_{\R^n}u^j(x,0)\,dx\right\}}{\max_{1\leq j\leq k}\left\{\int_{\R^n}u^j(x,0)\,dx\right\}}\geq\mu^0>0.
	\end{equation*}
	Then for $R>T^{\frac{1}{p}}>0$ there exists a constant $C^{\ast}=C^{\ast}\left(m,n\right)>0$ such that
	\begin{equation}\label{eq-claim-L-1-norm=of-u-at-initial-is-bounded-by-H-sub-m}
		\int_{\left\{|x|<R\right\}}u^l(x,0)\,dx\leq \frac{C^{\ast}}{\left(\mu^l\right)^{1+\frac{n(p-2)}{p}}}\left(\frac{R^{n+\frac{p}{p-2}}}{T^{\frac{1}{p-2}}}+T^{\frac{n}{p}}\left(u^l\right)^{1+\frac{n(p-2)}{p}}\left(0,T\right)\right) \qquad \forall 1\leq i\leq k
	\end{equation}
	where
	\begin{equation*}
		\mu^l=\frac{\int_{\R^n}u^l(x,0)\,dx}{\max_{1\leq j\leq k}\left\{\int_{\R^n}u^j(x,0)\,dx\right\}} \qquad \forall 1\leq i\leq k.
	\end{equation*}
\end{thm}

The result of Theorem \ref{lem-calculation-about-waiting-time-lower-bound} can be used to construct an initial trace of non-negative continuous weak solution $\bold{u}$ which belongs to a specific growth class.

\begin{cor}[cf. Theorem 4.1 of \cite{AC}]\label{cor-existence-of-initial-trace-from-Harnack-Type-Ineqlaity}
	Let $p>2$, $T>0$ and let $\bold{u}=\left(u^1,\cdots,u^k\right)$ be a continuous weak solution of 
	\begin{equation*}
		\left(u^l\right)_t=\nabla\left(\left|\nabla\bold{u}\right|^{p-2}\nabla u^l\right) \qquad \mbox{in $\R^n\times\left[0,T\right]$}.
	\end{equation*}
	Under the hypothesis of Theorem \ref{lem-calculation-about-waiting-time-lower-bound}, there exists a vector $\mbox{\boldmath$\rho$}=\left(\rho^1,\cdots,\rho^k\right)$ of nonnegative Borel measures on $\R^n$ such that
	\begin{equation}\label{result-1-of-existence-of-initila-trace}
		\lim_{t\to 0}\int_{\R^n}u^l(x,t)\vp(x)\,dx=\int_{\R^n}\vp(x)\,\rho^l\left(dx\right) \qquad \forall 1\leq l\leq k
	\end{equation}
	for all $\vp\in C_0\left(\R^n\right)$. Moreover, there exists a constant $C=C(n,m)>0$ such that
	\begin{equation}\label{result-2-of-existence-of-initila-trace}
		\int_{\left\{|x|<R\right\}}\rho^l(dx)<C\left(\frac{R^{n+\frac{p}{p-2}}}{T^{\frac{1}{p-2}}}+T^{\frac{n}{p}}\left(u^l\right)^{1+\frac{n(p-2)}{p}}\left(0,T\right)\right) \qquad\qquad  \forall 1\leq l\leq k
	\end{equation}
	for all $R>0$.
\end{cor}

\indent We end up this section by introducing the concept of weak solution. Let $E$ be an open set in $\R^n$, and for $T>0$ let $E_T$ denote the parabolic domain $E\times(0,T]$. We say that $\bold{u}=\left(u^1,\cdots,u^k\right)$ is a weak solution of \eqref{eq-main-equation-of-system} in $E_T$ if the component $u^l$, $\left(1\leq l\leq k\right)$, is a locally integrable function satisfying
\begin{enumerate}
\item $u^l$ belongs to function space:
\begin{equation*}\label{eq-first-condition-of-weak-soluiton-u-with-U}
\left|\nabla\bold{u}\right|^{p-2}\nabla u^l\in L^2\left(0,T:L^2\left(E\right)\right) 
\end{equation*}
\item $u^l$ satisfies the identity: 
\begin{equation}\label{eq-identity--of-formula-for-weak-solution}
\int_{0}^{T}\int_{E}\left\{\left|\nabla\bold{u}\right|^{p-2}\nabla u^l\cdot\nabla\vp-u^l\vp_t\right\}\,dxdt=0
\end{equation}
holds for any test function $\vp\in H_0^1\left(0,T:L^2\left(\R^n\right)\right)\cap L^2\left(0,T:W^{1,p}_0\left(\R^n\right)\right)$.
\end{enumerate}
\indent To control the regularity $\left(u^l\right)_t$, we consider the Lebesgue-Steklov average $\left(u^l\right)_h$ of the function $\left(u^l\right)$ which is introduced in \cite{D}: 
\begin{equation*}
	\left(u^l\right)_h(\cdot,t)=\frac{1}{h}\int_{t}^{t+h}u^l(\cdot,\tau)\,d\tau, \qquad \left(h>0\right).
\end{equation*}
Then, $\left(u^l\right)_h$ is well-defined and it converges to $u^l$ as $h\to 0$ in $L^{p}$ for all $p\geq 1$.  In addition, it is differentiable in time for all $h>0$ and its derivative is 
\begin{equation*}
	\frac{u^l(t+h)-u^l(t)}{h}.
\end{equation*}
Fix $t\in(0,T)$ and let $h$ be a small positive number such that $0<t<t+h<T$. Then for every compact subset $\mathcal{K}\subset\R^n$ the following formulation is equivalent to \eqref{eq-identity--of-formula-for-weak-solution}.
\begin{equation}\label{eq-formulation-for-weak-solution-of-u-h}
	\int_{\mathcal{K}\times\{t\}}\left[\left(\left(u^l\right)_h\right)_t\vp+m\left(\left|\nabla\bold{u}\right|^{p-2}\nabla u^l\right)_h\cdot\nabla\vp\right]\,dx=0, \qquad \forall 0<t<T-h
\end{equation}
for any $\vp\in H^1_0\left(\mathcal{K}\right)$. From now on, we use  the limit of \eqref{eq-formulation-for-weak-solution-of-u-h} as $h\to 0$ for the weak formulation \eqref{eq-identity--of-formula-for-weak-solution} of \eqref{eq-main-equation-of-system}.\\
\indent A brief outline of the paper is as follows. In Section 2, we will give preliminary results such as law of $L^1$-mass conservation, uniqueness and H\"older estimates of the solution $\bold{u}$. Section 3 is devoted to the existence of solution whose gradient is bounded globally in $L^{\infty}$. Section 4 deals with the asymptotic large time behaviour of solution with the method called entropy approach  (Theorem \ref{thm-convergence-between-absolute-bold-theta-and-fundamental-function-with-decay-rate}). Lastly, we consider a suitable Harnack type inequality of continuous weak solution of \eqref{eq-main-equation-of-system}. Through the last section, we show that the size of spatial average of the solution is controlled by the value of the solution at one point.

\section{Preliminary results}
\setcounter{equation}{0}
\setcounter{thm}{0}
In this section, we establish and prove various preliminary results which will play important roles throughout the coming section. The first preliminary result is the unique solution of \eqref{eq-main-equation-of-system}.

\begin{lemma}[\textbf{Uniqueness}]
	Let $p>2$, $n\geq 2$ and let $\bold{u}_1=\left(u_1^1,\cdots,u_1^k\right)$ and  $\bold{u}_2=\left(u^1_2,\cdots,u_2^k\right)$ be two solutions of \eqref{eq-main-equation-of-system} such that
	\begin{equation*}
		u_1^l,\,u_2^l\in L^p\left(0,T:W_0^{1,p}\left(\R^n\right)\right) \qquad \forall 1\leq l\leq k.
	\end{equation*}
	Suppose that 
	\begin{equation}\label{eq-assumption-equality-at-initial-time-p-La}
		u_1^l(x,0)=u_2^l(x,0) \qquad \forall x\in\R^n,\,\,1\leq l\leq k.
	\end{equation}
	Then, for any $T>0$
	\begin{equation}\label{eq-conclusion-of-equality-at-initial-time-p-La}
		\bold{u}_1(x,t)=\bold{u}_2(x,t) \qquad \forall (x,t)\in\R^n\times[0,T).
	\end{equation}
\end{lemma}

\begin{proof}
	Consider the  weak formulation \eqref{eq-identity--of-formula-for-weak-solution} for $u^l_1-u^l_2$. Taking $u^l_1-u^l_2$ as a test function and summing it over $i=1$, $\cdots$, $k$, we can have
	\begin{equation}\label{eq-weak-energy-type-inequality-for-comparison-first-1}
		\begin{aligned}
			&\sum_{i=1}^{k}\left[\sup_{0\leq t\leq T}\int_{\R^n}\left(u_1^l-u_2^l\right)^2\left(x,t\right)\,dx\right]\\
			&\qquad \qquad +\int_{0}^{T}\int_{\R^n}\left(\left|\overline{\bold{u}}_1\right|^{p-2}\overline{\bold{u}}_1-\left|\overline{\bold{u}}_2\right|^{p-2}\overline{\bold{u}}_2\right)\cdot\left(\overline{\bold{u}}_1-\overline{\bold{u}}_2\right)\,dxdt \leq \sum_{i=1}^{k}\left[\int_{\R^n}\left(\overline{u}_0^l-u_0^l\right)_+^2\left(x\right)\,dx\right]
		\end{aligned}
	\end{equation}
	where $\overline{\bold{u}}_1$ are $\overline{\bold{u}}_2$ two vectors such that
	\begin{equation*}
		\overline{\bold{u}}_1=\left(\nabla u^1_1,\cdots,\nabla u^k_1\right)	\qquad \mbox{and} \qquad \overline{\bold{u}}_2=\left(\nabla u^1_2,\cdots,\nabla u^k_2\right).	
	\end{equation*}
	Let $\alpha$ be the angle between $\overline{\bold{u}}_1$ and $\overline{\bold{u}}_2$ which is less than $\pi$ and suppose that both $\left|\overline{\bold{u}}_1\right|$ and $\left|\overline{\bold{u}}_2\right|$ are nonzero. Then, by the theory of vector calculus we can have
	\begin{equation*}
		\begin{aligned}
			\left(\left|\overline{\bold{u}}_1\right|^{p-2}\overline{\bold{u}}_1-\left|\overline{\bold{u}}_2\right|^{p-2}\overline{\bold{u}}_2\right)\cdot\left(\overline{\bold{u}}_1-\overline{\bold{u}}_2\right)&=\left(\left|\overline{\bold{u}}_1\right|^{p-2}\overline{\bold{u}}_1+\left|-\overline{\bold{u}}_2\right|^{p-2}\left(-\overline{\bold{u}}_2\right)\right)\cdot\left(\overline{\bold{u}}_1+\left(-\overline{\bold{u}}_2\right)\right)\\
			&\geq \left|\left|\overline{\bold{u}}_1\right|^{p-2}\overline{\bold{u}}_1+\left|-\overline{\bold{u}}_2\right|^{p-2}\left(-\overline{\bold{u}}_2\right)\right|\left|\overline{\bold{u}}_1+\left(-\overline{\bold{u}}_2\right)\right|\cos\frac{\alpha}{2}\geq 0.
		\end{aligned}
	\end{equation*}
	Thus
	\begin{equation}\label{eq-positivity-of-second-integral-in-weak-engery-thpei-inequ-for-comparison}
		\begin{aligned}
			&\int_{0}^{T}\int_{\R^n}\left(\left|\overline{\bold{u}}_1\right|^{p-2}\overline{\bold{u}}_1-\left|\overline{\bold{u}}_2\right|^{p-2}\overline{\bold{u}}_2\right)\cdot\left(\overline{\bold{u}}_1-\overline{\bold{u}}_2\right)\,dxdt\\
			&\qquad \qquad \geq \int_{0}^{T}\int_{\R^n\cap\left\{\left|\bold{u}_1\right|>0\right\}\cap\left\{\left|\bold{u}_2\right|>0\right\}}\left|\left|\overline{\bold{u}}_1\right|^{p-2}\overline{\bold{u}}_1+\left|-\overline{\bold{u}}_2\right|^{p-2}\left(-\overline{\bold{u}}_2\right)\right|\left|\overline{\bold{u}}_1+\left(-\overline{\bold{u}}_2\right)\right|\cos\frac{\alpha}{2}\,dxdt\geq 0.
		\end{aligned}
	\end{equation}
	By \eqref{eq-weak-energy-type-inequality-for-comparison-first-1}, \eqref{eq-positivity-of-second-integral-in-weak-engery-thpei-inequ-for-comparison} and initial condition, we have 
	\begin{equation}\label{eq-L-2-contraction-between-two-vectors-nabla-bold-i-1-and-nabla-bold-u-2}
		\sum_{i=1}^{k}\left[\sup_{0\leq t\leq T}\int_{\R^n}\left(\overline{u}^l-u^l\right)_+^2\left(x,t\right)\,dx\right]\leq 	\sum_{i=1}^{k}\left[\int_{\R^n}\left(\overline{u}_0^l-u_0^l\right)_+^2\left(x\right)\,dx\right]=0
	\end{equation}
	and the lemma follows.
\end{proof}

The second preliminary result is the Law of $L^1$ Mass Conservation on \eqref{eq-main-equation-of-system}. In the mathematical theory of PDE, the concept of mass conservation plays an important role in the study of asymptotic large time behaviour of solutions.

\begin{lemma}[\textbf{Law of $L^1$ mass conservation}]\label{lem-Law-of-L-1-Mass-Conservation}
	Let $n\geq 2$, $p>2$ and let $\bold{u}_0\in L^2\left(\R^n\right)$. Let $\bold{u}=\left(u^1,\cdots,u^k\right)$ be a weak solution of \eqref{eq-main-equation-of-system} in $\R^n\times(0,\infty)$. Then, for any $t>0$
	\begin{equation}\label{eq-in-Lemma-for-mass-conservation-of-component-01}
		\int_{\R^n}u^l(x,t)\,dx=\int_{\R^n}u^l_0(x)\,dx.
	\end{equation}
\end{lemma}
\begin{proof}
	By weak formulation of \eqref{eq-main-equation-of-system}, we can get
	\begin{equation}\label{eq-estimates-of-L-2-of-u-i-and}
		\frac{1}{2}\sup_{0<t<T}\int_{\R^n}\left(u^l(x,t)\right)^{2}\,dx+\int_{0}^{T}\int_{\R^n}\left|\nabla \bold{u}\right|^{p-2}\left|\nabla u^l\right|^2\,dxdt\leq \int_{\R^n}\left(u_{0}^l\right)^{2}\,dx \qquad \forall 1\leq l\leq k.
	\end{equation}
	Let $\left\{\zeta_j\right\}\subset W^{1,p}(\R^n)$ be a sequence of cut-off functions such that
	\begin{equation*}
		\zeta_j(x)=1 \quad \mbox{for $|x|\leq j$}, \qquad \zeta_j(x)=0 \quad \mbox{for $|x|\geq 2j$}, \qquad 0<\zeta_j(x)<1 \quad \mbox{for $j<|x|<2j$}.
	\end{equation*}
	Let $w=\left|\nabla\bold{u}\right|$. Then, by \eqref{eq-estimates-of-L-2-of-u-i-and} we have
	\begin{equation*}
		\begin{aligned}
			\left|\int_{\R^n}u^l(x,t)\zeta_j(x)\,dx-\int_{\R^n}u_0^l(x)\zeta_j(x)\,dx\right|&=\left|\int_{0}^{t}\int_{\R^n}\left(u^l\right)_t\zeta_j\,dxdt\right|\\
			&=\left|-\int_{0}^{t}\int_{\R^n}w^{p-2}\nabla u^l\cdot\nabla \zeta_j\,dxdt\right|\\
			& \leq t^{\frac{1}{p}}\,\left\|\nabla\zeta_0\right\|^{\frac{1}{p}}_{L^{p}\left(\R^n\right)}\left(\int_{0}^{t}\int_{B_{2j}\bs B_{j}}w^{p}\,dxdt\right)^{\frac{p-1}{p}}\\
			& \leq t^{\frac{1}{p}}\,\left\|\nabla\zeta_0\right\|^{\frac{1}{p}}_{L^{p}\left(\R^n\right)}\left(\int_{B_{2R}\bs B_{R}}\left|\nabla\bold{u}_0\right|^{2}\,dx\right)^{\frac{p-1}{p}}\to 0 \qquad \mbox{as $j\to\infty$}.
		\end{aligned}
	\end{equation*}
	Thus \eqref{eq-in-Lemma-for-mass-conservation-of-component-01} holds and the lemma follows. 
\end{proof}

We are wrapping up this section by introducing the local continuity of the weak solution of \eqref{eq-main-equation-of-system}, especially the H\"older estimates of the solution $\bold{u}$ and of the gradient of solution $\nabla\bold{u}$. The proofs of following estimates can be found in \cite{M} (Theorem 1) and  \cite{D} (Theorem 1.1 of Chap. IX).
\begin{lemma}[Local H\"older Estimates]
	Let $p>2$. Then any weak solution $\bold{u}$ of \eqref{eq-main-equation-of-system} is locally H\"older continuous in $\R^n\times\left(0,T\right]$. Moreover, if $\left\|\nabla\bold{u}\right\|_{L^{\infty}\left(\R^n\times\left(0,T\right]\right)}<\infty$ then the function $u^l_{x_j}$ is also locally H\"older continuous in $\R^n\times\left(0,T\right]$ for all $1\leq i\leq k$ and $1\leq j\leq n$.
\end{lemma}

\section{Uniform Boundedness of the diffusion coefficinets $\left|\nabla \bold{u}\right|$}
\setcounter{equation}{0}
\setcounter{thm}{0}

Since the evolution of solutions are influenced mostly by the diffusion coefficients and external force as time goes by, it is very important to get the suitable regularity estimates for the diffusion coefficients. This section is devoted to provide the existence of solution $\bold{u}$ and a priori $L^{\infty}$ boundedness of diffusion coefficients $\left|\nabla\bold{u}\right|$. The proof of $L^{\infty}$ boundedness is based on a recurrence relation between a series of truncation of $\left|\nabla \bold{u}\right|$. An energy type inequality and an embedding play important roles on the proof.\\
\indent We start this section by stating well-known inequality.

\begin{lemma}[cf. Proposition 3.1 of Chap. I of \cite{D}]\label{Proposition-3-1-of-Chap-I-of-cite-D}
	Let $m_1$, $m_2\geq 1$ and let $0\leq t<T$. There exists a constant $C>0$ depending on $n$ and $p$ such that for every $v\in L^{\infty}\left(t,T;L^{m_1}\left(\R^n\right)\right)\cap L^{m_2}\left(t,T;W_0^{1,m_2}\left(\R^n\right)\right)$
	\begin{equation*}
		\int_{t}^{T}\int_{\R^n}\left|v\right|^{\frac{m_2\left(n+m_1\right)}{n}}\,dxdt\leq C\left(\int_{t}^{T}\int_{\R^n}\left|\nabla v\right|^{m_2}\,dxdt\right)\left(\sup_{t<\tau<T}\int_{\R^n}\left|v(\cdot,\tau)\right|^{m_1}\,dx\right)^{\frac{m_2}{n}}.
	\end{equation*}
\end{lemma}

We now are ready to prove the Theorem \ref{thm-Main-boundedness-of-diffusion-coefficients-of-p-Laplacian}.

\begin{proof}[\textbf{Proof of Theorem \ref{thm-Main-boundedness-of-diffusion-coefficients-of-p-Laplacian}}]
		For each $1\leq l\leq k$, let $\left\{u^l_{0,\,\widehat{j}}\right\}\subset C_0^{\infty}\left(\R^n\right)$ be a sequence of functions such that
	\begin{equation}\label{eq-initial-condition-of-u-l-widehat-j-01}
		\left|\nabla u^l_{0,\,\widehat{j}}\right|\leq \,\widehat{j}\,,\qquad \left\|u^l_{0,\,\widehat{j}}\right\|_{L^2\left(\R^n\right)}\leq 2\left\|u_{0}\right\|_{L^2\left(\R^n\right)},\qquad  \left\|u^l_{0,\,\widehat{j}}\right\|_{W^{1,p}\left(\R^n\right)}\leq 2\left\|u_{0}\right\|_{W^{1,p}\left(\R^n\right)}    \qquad \forall \,\widehat{j}\in\N 
	\end{equation}
	and
	\begin{equation*}\label{eq-convergience-between-smooth-soluton-u-widehat-j-and-u}
		u_{0,\,\widehat{j}}\to u_{0} \qquad \mbox{in $L^2\cap W^{1,p}$ as $\,\widehat{j}\to\infty$}.
	\end{equation*}
	By the standard theory for the uniformly parabolic equation \cite{LSU}, a smooth solution $\bold{u}_{\,\widehat{j}}=\left(u^1_{\,\widehat{j}},\cdots,u^k_{\,\widehat{j}}\right)$ of the system 
	\begin{equation}\label{eq-for-existence-non-degenerate-pde-with-bold-w-epsilon-M-diffusion-coefficients}
		\begin{cases}
			\begin{aligned}
				\left(u^l_{\,\widehat{j}}\right)_t&=\nabla\cdot\left(\left(\left|\nabla \bold{u}_{\,\widehat{j}}\right|^{p-2}+\frac{1}{\widehat{j}}\right)\nabla u^l_{\,\widehat{j}}\right) \qquad \mbox{in $\R^n\times\left(0,T\right]$}\\
				u^l_{\,\widehat{j}}\,(x,0)&=u^l_{0,\,\widehat{j}}\,(x) \qquad \qquad \qquad \qquad  \forall x\in \R^n.
			\end{aligned}
		\end{cases}
	\end{equation}
	exists on a short time interval. Let $(0,t_0)$ be the maximal interval of existence of smooth solution and suppose that $t_0<\infty$.\\
	
	\textbf{Claim 1:} $\left|\nabla\bold{u}_{\widehat{j}}\right|$ is $L^{\infty}$ bounded at $t=t_0$.\\
	We will use a modification of the proof of Theorem 1 of \cite{CV} to prove the claim. Let 
	\begin{equation}\label{eq-definition-overline-w-widehat-j}
		\overline{w}_{\,\widehat{j}}=\left|\nabla\bold{u}_{\,\widehat{j}}\right|^2=\sum_{l=1}^k\left|\nabla u_{\,\widehat{j}}^l\right|^2=\sum_{i=1}^n\sum_{l=1}^{k}\left[\left(u_{\,\widehat{j}}^l\right)_{x_i}\right]^2=\sum_{i=1}^n\sum_{l=1}^{k}\left[\left(u_{\,\widehat{j}}^l\right)_{x_i}\right]^2.
	\end{equation}
	Then, by simple computation we can get
	\begin{equation}\label{eq-for-bold-v-diffusion-coefficients-of-p-La-direct-computiation}
		\left(\overline{w}_{\,\widehat{j}}\right)_t=\left[\left(\delta_{ij}\overline{w}_{\,\widehat{j}}^{\frac{p-2}{2}}+\left(p-2\right)\overline{a}_{\,\widehat{j}}^{ij}\overline{w}_{\,\widehat{j}}^{\frac{p-4}{2}}\right)\left(\overline{w}_{\,\widehat{j}}\right)_{x_i}\right]_{x_j}+\frac{1}{\,\widehat{j}}\left(\La\overline{w}_{\,\widehat{j}}\,\,-2\sum_{l=1}^k\sum_{i=1}^{n}\left|\nabla \left(u_{\,\widehat{j}}^l\right)_{x_i}\right|^2\right)-\overline{F} \qquad \forall 1\leq i, \,j\leq n
	\end{equation}
	where
	\begin{equation*}
		\overline{a}_{\,\widehat{j}}^{ij}(x,t)=\sum_{l=1}^k\left(u^l_{\,\widehat{j}}\right)_{x_i}\left(u^l_{\,\widehat{j}}\right)_{x_j} \qquad\mbox{and} \qquad \overline{F}=2\sum_{l=1}^k\sum_{i=1}^{n}\overline{w}_{\,\widehat{j}}^{\frac{p-2}{2}}\left|\nabla \left(u_{\,\widehat{j}}^l\right)_{x_i}\right|^2+\frac{p-2}{2}\overline{w}_{\,\widehat{j}}^{\frac{p-4}{2}}\left|\nabla \overline{w}_{\,\widehat{j}}\right|^2.
	\end{equation*}
	The matrix
	\begin{equation*}
		\left(\begin{array}{c}\left(u^l_{\,\widehat{j}}\right)_{x_1}\\ \vdots \\\left(u^l_{\,\widehat{j}}\right)_{x_n}\end{array}\right)\left(\left(u^l_{\,\widehat{j}}\right)_{x_1},\cdots, \left(u^l_{\,\widehat{j}}\right)_{x_n}\right)=\left(\nabla u^l_{\,\widehat{j}}\right)^T\nabla u^l_{\,\widehat{j}} \qquad \forall 1\leq l\leq k,
	\end{equation*}
	is symmetric and positive semi-definite having exactly $n-1$ zero eigenvalues. Moreover, the remaining eigenvalue is 
	\begin{equation*}
		\lambda_{\,\widehat{j}}=\nabla u^l_{\,\widehat{j}}\left(\nabla u^l_{\,\widehat{j}}\right)^T=\left|\nabla u^l_{\,\widehat{j}}\right|^2 \qquad 1\leq l\leq k.
	\end{equation*}
	Thus the matrix $\left[\overline{a}_{\,\widehat{j}}^{ij}\right]_{1\leq i,\,j\leq n}$ is also a symmetric and positive semi-definite satisfying
	\begin{equation}\label{eq-property-of-matrix-overline-a-i-j}
		0\leq \overline{a}_{\,\widehat{j}}^{ij}\,\xi_i\xi_j\leq \overline{w}_{\,\widehat{j}}\left|\xi\right|^2 \qquad \forall \xi\in\R^n.
	\end{equation}
	By \eqref{eq-for-bold-v-diffusion-coefficients-of-p-La-direct-computiation} and \eqref{eq-property-of-matrix-overline-a-i-j}, $\overline{w}_{\,\widehat{j}}$ satisfies
	\begin{equation}\label{eq-for-bold-w-with-positive-definite-a-ij}
		\left(\overline{w}_{\,\widehat{j}}\right)_t=\left[a_{\,\widehat{j}}^{ij}\left(\overline{w}_{\,\widehat{j}}^{\frac{p}{2}}\right)_{x_i}\right]_{x_j}+\frac{1}{\,\widehat{j}}\left(\La\overline{w}_{\,\widehat{j}}\,\,-2\sum_{l=1}^k\sum_{i=1}^{n}\left|\nabla \left(u_{\,\widehat{j}}^l\right)_{x_i}\right|^2\right)-\overline{F} \qquad \forall (x,t)\in\R^n\times(0,\infty)
	\end{equation}
	for the positive definite matrix $\left[a_{\,\widehat{j}}^{ij}\right]_{1\leq i,\,j\leq n}=\frac{2}{p}\left[\delta_{ij}+\left(p-2\right)\frac{\sum_{l=1}^k\left(u_{\,\widehat{j}}^l\right)_{x_i}\left(u_{\,\widehat{j}}^l\right)_{x_j}}{\overline{w}_{\,\widehat{j}}}\right]_{1\leq i,\,j\leq n}$ satisfying
	\begin{equation}\label{eq-uniformly-ellitic-condition-of-diffusion-coefficient-a-i-j}
		\frac{2}{p}\left|\xi\right|^2\leq a_{\,\widehat{j}}^{ij}\,\xi_i\xi_j\leq \frac{2(p-1)}{p}\left|\xi\right|^2 \qquad \forall \xi=\left(\xi_1,\cdots,\xi_n\right)\in\R^n.
	\end{equation}
	Let $\overline{w}_{\,\widehat{j}}=w_{\,\widehat{j}}^2\,$. Then
	\begin{equation}\label{eq-degenerate-pde-of-w-absolute-value-of-diffusion-coefficient-of-p-Laplacian}
		\left(w_{\,\widehat{j}}\right)_t=\frac{p}{2(p-1)}\left[a_{\,\widehat{j}}^{ij}\left(w_{\,\widehat{j}}^{\,p-1}\right)_{x_i}\right]_{x_j}+\frac{1}{\,\widehat{j}}\,\,\La w_{\,\widehat{j}}+\frac{1}{\,\widehat{j}\,\,w_{\,\widehat{j}}}\left(\left|\nabla w_{\,\widehat{j}}\right|^2\,\,-\sum_{l=1}^k\sum_{i=1}^{n}\left|\nabla \left(u_{\,\widehat{j}}^l\right)_{x_i}\right|^2\right)+F
	\end{equation}
	where
	\begin{equation*}
		\begin{aligned}
			F=\frac{p}{2}w_{\,\widehat{j}}^{\,p-3}a_{\,\widehat{j}}^{ij}\left(w_{\,\widehat{j}}\right)_{x_i}\left(w_{\,\widehat{j}}\right)_{x_j}-\sum_{l=1}^k\sum_{i=1}^{n}w_{\,\widehat{j}}^{p-3}\left|\nabla \left(u_{\,\widehat{j}}^l\right)_{x_i}\right|^2-(p-2)w_{\,\widehat{j}}^{p-3}\left|\nabla w_{\,\widehat{j}}\right|^2.
		\end{aligned}
	\end{equation*}
	Observe that, by Cauchy-Schwarz inequality we have
	\begin{equation}\label{negativity-of-w-widehat-j-from-C-S-ineqiatly-01}
	\left|\nabla w_{\,\widehat{j}}\right|^2-\sum_{l=1}^k\sum_{i=1}^{n}\left|\nabla \left(u_{\,\widehat{j}}\right)^l_{x_i}\right|^2\leq 0.
	\end{equation}
	For $s\in\N$, let
	\begin{equation*}
		L_s=M\left(1-\frac{1}{2^{s}}\right) \qquad \mbox{and} \qquad \left(w_{\,\widehat{j}}\right)_s=\left(w_{\,\widehat{j}}-L_s\right)_+
	\end{equation*}
	for a constant $M>2$ which will be determined later. Then
	\begin{equation}\label{eq-lower-bound-of-U-on-U-j-positibve}
		w_{\,\widehat{j}}\geq \frac{M}{2}>1 \qquad \mbox{on $\left\{\left(w_{\,\widehat{j}}\right)_s\geq 0\right\}$} \qquad \forall j\in\N.
	\end{equation}
	Multiply \eqref{eq-degenerate-pde-of-w-absolute-value-of-diffusion-coefficient-of-p-Laplacian} by $\left(w_{\,\widehat{j}}\right)^{p-1}_s$ and integrate it over $\R^n$. Then, by \eqref{eq-uniformly-ellitic-condition-of-diffusion-coefficient-a-i-j} and \eqref{negativity-of-w-widehat-j-from-C-S-ineqiatly-01} we have the following energy inequality for $\left(w_{\,\widehat{j}}\right)_s$:
	\begin{equation}\label{eq-energy-type-inequality-for-boundedness-of-w-l}
	\begin{aligned}
	\frac{\partial}{\partial t}\left(\frac{1}{p}\int_{\R^n}\left(w_{\,\widehat{j}}\right)_s^{p}\,dx\right)+\frac{1}{p-1}\int_{\R^n}\left|\nabla \left(w_{\,\widehat{j}}\right)_s^{p-1}\right|^2\,dx\leq 0.
	\end{aligned}
	\end{equation}
	For fixed $0<t_1<t_0$, let $T_s=t_1\left(1-\frac{1}{2^{ps}}\right)$ and 
	\begin{equation*}
		A_s=\frac{1}{p}\sup_{T_s\leq t\leq t_0}\left(\int_{\R^n}\left(w_{\,\widehat{j}}\right)_s^{p}\,dx\right)+\frac{1}{p-1}\int_{T_s}^{t_0}\int_{\R^n}\left|\nabla \left(w_{\,\widehat{j}}\right)_s^{p-1}\right|^2\,dxdt.
	\end{equation*}
	Integrating \eqref{eq-energy-type-inequality-for-boundedness-of-w-l} over $\left(\tau,t\right)$ and $\left(\tau,t_0\right)$, $\left(T_{s-1}<\tau<T_s,\,\,T_s<t<t_0\right)$, we have
	\begin{equation}\label{eq-upper-bound-of-A-s-from-energy-ibnequaity}
		\begin{aligned}
			A_s\leq \int_{\R^n}\left(w_{\,\widehat{j}}\right)_s^{p}\left(x,\tau\right)\,dx.
		\end{aligned}
	\end{equation}
	Taking the mean value in $\tau$ on $\left[T_{s-1},T_s\right]$, we have
	\begin{equation}\label{eq-inequality-with-A-j-2}
		\begin{aligned}
			A_s\leq \frac{2^{ps}}{t_1}\int_{T_{s-1}}^{\infty}\int_{\R^n}\left(w_{\,\widehat{j}}\right)_s^{p}\,dxdt.
		\end{aligned}
	\end{equation}
	By Lemma \ref{Proposition-3-1-of-Chap-I-of-cite-D} with $v$, $m_1$ and $m_2$ being replaced by $\left(w_{\,\widehat{j}}\right)_s^{p-1}$, $\frac{p}{p-1}$ and $2$, respectively, we can get
	\begin{equation}\label{eq-Sobolev-and-Interpolation-inequalities-for-replacign-U-by-U-j}
		\begin{aligned}
			A_{s-1}\geq C_1\left(\int_{T_{s-1}}^{\infty}\int_{\R^n}\left(w_{\,\widehat{j}}\right)_{s-1}^{2\left(p-1+\frac{p}{n}\right)}\,dxdt\right)^{\frac{1}{2\left(p-1+\frac{p}{n}\right)}}
		\end{aligned}
	\end{equation}
	for some constant $C_1>0$. Since $\left(w_{\,\widehat{j}}\right)_{s-1}\geq \frac{M}{2^s}$ on $\left(w_{\,\widehat{j}}\right)_s\geq 0$, we can get
	\begin{equation}\label{eq-lower-bound-of-w-j-1-on-w-j-positive-1}
	\chi_{\left\{\left(w_{\,\widehat{j}}\right)_s\geq 0\right\}}\leq\left(\frac{2^s}{M}\left(w_{\,\widehat{j}}\right)_{s-1}\right)^{p-2+\frac{2p}{n}}.
	\end{equation}
	By \eqref{eq-inequality-with-A-j-2},  \eqref{eq-Sobolev-and-Interpolation-inequalities-for-replacign-U-by-U-j} and \eqref{eq-lower-bound-of-w-j-1-on-w-j-positive-1}, 
	\begin{equation}\label{eq-inequality-with-A-j-4}
		\begin{aligned}
			A_s\leq \frac{2^{2\left(p-1+\frac{p}{n}\right)s}}{t_1M^{p-2+\frac{2p}{n}}}\int_{T_{j-1}}^{\infty}\int_{\R^n}\left(w_{\,\widehat{j}}\right)_{s-1}^{2\left(p-1+\frac{p}{n}\right)}\,dxdt\leq \frac{C_2\,4^{\left(p-1+\frac{p}{n}\right)s}}{t_1M^{p-2+\frac{2p}{n}}}A_{s-1}^{1+\left(2p-3+\frac{2p}{n}\right)}.
		\end{aligned}
	\end{equation}
	for some constant $C_2>0$. Let $M=\frac{\textbf{M}}{t_1^{\frac{n}{n(p-2)+2p}}}$ and choose the constant $\textbf{M}>0$ so large that
	\begin{equation}\label{eq-condition-for-constant-bold-M-from-A-0}
		A_0\leq \left(\frac{\textbf{M}^{p-2+\frac{2p}{n}}}{C_2}\right)^{\frac{1}{2p-3+\frac{2p}{n}}}4^{-\left(p-1+\frac{p}{n}\right)\left(\frac{1}{2p-3+\frac{2p}{n}}\right)^2}.
	\end{equation}
	Then, by Lemma 4.1 of Chap. I of \cite{D} we have 
	\begin{equation*}
		A_s\to 0\qquad \mbox{ as $s\to\infty$}.
	\end{equation*}
	Therefore,
	\begin{equation}\label{eq-uniform-boundednes-of-gradient-of-solution-bold-u-widehat-j}
		\sup_{x\in\R^n,\,t_1\leq t\leq t_0}\left|\nabla\bold{u}_{\,\widehat{j}}\left(x,t\right)\right|\leq \frac{\textbf{M}}{t_1^{\frac{n}{n(p-2)+2p}}}
	\end{equation} 
	and the \textbf{claim 1} follows.\\
	
	\indent By the \textbf{claim 1}, the equation for $u^l_{\,\widehat{j}}$ is still uniformly parabolic at $t=t_0$. Thus, by the standard theory for the uniformly parabolic equation, we can extend the smoothness of $u^l_{\,\widehat{j}}$ to the time interval $\left(0,t_0+t_2\right)$ for some constant $t_2>0$. By the maximality of $t_0$, a contradiction arises. Hence $t_0=\infty$.\\
	\indent Let $T>0$. By simple computation, it can be easily checked that
	\begin{equation}\label{eq-estimates-of-L-2-of-u-i-and-H-1-of--sqrt-epsilon-u-i}
		\begin{aligned}
			&\frac{1}{2}\sup_{0<t<T}\int_{\R^n}\left(u_{\,\widehat{j}}^l(x,t)\right)^{2}\,dx+\int_{0}^{T}\int_{\R^n}\left|\nabla \bold{u}_{\,\widehat{j}}\right|^{p-2}\left|\nabla u^l_{\,\widehat{j}}\right|^2\,dxdt\\
			&\qquad \qquad \qquad \qquad \qquad \qquad+\frac{1}{j}\int_{0}^{T}\int_{\R^n}\left|\nabla u^l_{j}\right|^2\,dxdt \leq \int_{\R^n}\left(u_{0}^l\right)^{2}\,dx \qquad \forall 1\leq l\leq k.
		\end{aligned}
	\end{equation}
	By \eqref{eq-estimates-of-L-2-of-u-i-and-H-1-of--sqrt-epsilon-u-i}, $\left\{u^{l}_{\,\widehat{j}}\right\}$ is uniformly bounded in $L^2\left(\R^n\times\left(0,T\right)\right)$ and it has a subsequence which we may assume without loss of generality to be the sequence itself that  converges to some function $u^l$ weakly in $L^{2}\left(\R^n\times\left(0,T\right)\right)$ as $\widehat{j}\to\infty$. Moreover,
	\begin{equation}\label{eq-L-1-loc-convergence-of-u-i-and-bold-w}
		\begin{cases}
			\begin{array}{clll}
				\nabla u^l_{\,\widehat{j}}\to \nabla u^l &\qquad&&\\
				&\qquad &&\mbox{weakly in $L^{p}\left(\R^n\times\left(0,T\right)\right)$\qquad as $\widehat{j}\to\infty$}\\
				\left|\nabla \bold{u}_{\,\widehat{j}}\right|\to \left|\nabla \bold{u}\right|=\sqrt{\sum_{i=1}^{k}\left(\nabla u^l\right)^2}&\qquad&&
			\end{array}
		\end{cases}
	\end{equation}
	and
	\begin{equation}\label{eq-L-1-loc-convergence-of-u-i-and-bold-w-with-1-over-j}
		\frac{1}{\widehat{j}}\,\nabla u_{\,\widehat{j}}^l\to 0 \qquad \qquad \mbox{weakly in $L^{2}\left(\R^n\times\left(0,T\right)\right)$\qquad as $\widehat{j}\to\infty$}.
	\end{equation}
	Let $\vp\in H_0^1\left(0,T:L^2\left(\R^n\right)\right)\cap L^2\left(0,T:W^{1,p}_0\left(\R^n\right)\right)$ be a test function. Then, by the weak formulation of  \eqref{eq-for-existence-non-degenerate-pde-with-bold-w-epsilon-M-diffusion-coefficients} we have
	\begin{equation}\label{eq-for-weak-concept-of-solution-type-1}
		\begin{aligned}
			\int_{0}^{\infty}\int_{\R^n}\left|\nabla \bold{u}_{\,\widehat{j}}\right|^{\,p-2}\nabla u_{\,\widehat{j}}^l\cdot\nabla\vp\,dxdt+\int_{0}^{\infty}\int_{\R^n}\frac{1}{\widehat{j}}\nabla u_{\,\widehat{j}}^l\cdot\nabla\vp\,dxdt-\int_{0}^{\infty}\int_{\R^n}u^l_{\,\widehat{j}}\,\vp_t\,dxdt=0 \qquad \forall 1\leq l\leq k.
		\end{aligned}
	\end{equation}
	Letting $\widehat{j}\to \infty$ in \eqref{eq-for-weak-concept-of-solution-type-1}, by \eqref{eq-L-1-loc-convergence-of-u-i-and-bold-w} and \eqref{eq-L-1-loc-convergence-of-u-i-and-bold-w-with-1-over-j} we have
	\begin{equation*}\label{eq-for-weak-concept-of-solution-type-1-after-limits-0}
		\int_{0}^{\infty}\int_{\R^n}\left|\nabla \bold{u}\right|^{\,p-2}\nabla u^l\cdot\nabla\vp\,dxdt -\int_{0}^{\infty}\int_{\R^n}u^l\vp_t\,dxdt=0\qquad \forall 1\leq l\leq k,
	\end{equation*}
	and the weak formulation of \eqref{eq-main-equation-of-system} follows.\\
	\indent We now are going to show that 
	\begin{equation}\label{eq-converges-of-u-to-initial-data-as-t-to-zero}
		u^l\left(\cdot,t\right)\to u^l_0\qquad  \mbox{in $L^1$ as $t\to 0^+$},\,\,\forall 1\leq l\leq k.
	\end{equation} 
	Let $\eta(x)\in C_0^{2}\left(\R^n\right)$ and let $0<t<1$. Then, by \eqref{eq-estimates-of-L-2-of-u-i-and-H-1-of--sqrt-epsilon-u-i} we have
	\begin{equation}\label{compare-between-u-and-u-0-with-eta23579}
		\begin{aligned}
			&\left|\int_{\R^n}u^l_{\,\widehat{j}}(x,t)\eta(x)\,dx-\int_{\R^n}u^l_{0,\,\widehat{j}}(x)\eta(x)\,dx\right|\\
			&\qquad \qquad  \leq \int_{0}^{t}\int_{\R^n}\left|\nabla \bold{u}_{\,\widehat{j}}\right|^{\,p-2}\left|\nabla u^l_{\,\widehat{j}}(x,t)\right|\left|\nabla\eta(x)\right|\,dxdt+\frac{1}{\widehat{j}}\int_{0}^{t}\int_{\R^n}\left|\nabla u^l_{\,\widehat{j}}(x,t)\right|\left|\nabla\eta(x)\right|\,dxdt\\
			&\qquad \qquad \leq k^{\frac{p-2}{2}}\sum_{i=1}^{k}\int_{0}^{t}\int_{\R^n}\left|\nabla u^l_{\,\widehat{j}}(x,t)\right|^{p-1}\left|\nabla\eta(x)\right|\,dxdt+\frac{1}{\widehat{j}}\int_{0}^{t}\int_{\R^n}\left|\nabla u^l_{\,\widehat{j}}(x,t)\right|\left|\nabla\eta(x)\right|\,dxdt\\
			&\qquad \qquad \leq C\left(\left\|\bold{u}_0\right\|_{L^{2}\left(\R^n\right)},\left\|\nabla\eta\right\|_{L^{\infty}}\right)\left(t^{\frac{1}{p}}+\left(\frac{t}{\,\widehat{j}}\right)^{\frac{1}{2}}\right),\qquad \qquad \forall 0<t<1,\,\,1\leq l\leq k.
		\end{aligned}
	\end{equation}
	Thus, letting $\widehat{j}\to 0$ and then $t\to 0$ in \eqref{compare-between-u-and-u-0-with-eta23579}, \eqref{eq-converges-of-u-to-initial-data-as-t-to-zero} holds and $\bold{u}=\left(u^1,\cdots,u^k\right)$ is a weak solution of \eqref{eq-main-equation-of-system}.\\
\indent To complete the proof, we will show the $L^{\infty}$ boundedness of diffusion coefficients $\left|\nabla\bold{u}\right|$. By \eqref{eq-initial-condition-of-u-l-widehat-j-01}, \eqref{eq-upper-bound-of-A-s-from-energy-ibnequaity} and \eqref{eq-condition-for-constant-bold-M-from-A-0}, we can take the constant $\bold{M}$ in \eqref{eq-uniform-boundednes-of-gradient-of-solution-bold-u-widehat-j} to be only controlled by $\left\|\nabla \bold{u}_0\right\|_{L^{p}}$. Thus, letting $\widehat{j}\to\infty$ in \eqref{eq-uniform-boundednes-of-gradient-of-solution-bold-u-widehat-j} we have the inequality \eqref{eq-uniform-boundeness-of-grad-of-soluiton-bold-u} and the theorem follows.
\end{proof}

\begin{remark}
	With the weak differentiability of $\left|\nabla \bold{u}\right|^{\frac{p-2}{2}}u^l_{x_j}$, $\left(1\leq l\leq k,\,\,1\leq j\leq n\right)$, $L_{loc}^{p}$-boundedness of $\left|\nabla u\right|$ was studied by E. DiBenedetto. We refer the reader to the Chap. VIII of \cite{D} for the local boundedness of the gradient of solution of \eqref{eq-main-equation-of-system}.
\end{remark}

\section{Asymptotic Large Time Behaviour: The Entropy Approach}
\setcounter{equation}{0}
\setcounter{thm}{0}

In this section, we will investigate the convergence between the solution of \eqref{eq-main-equation-of-system} and the fundamental solution of $p$-Laplacian equation. We use a modification of the techniques used in \cite{A},  \cite{DD} and \cite{V} to show the convergence. For the constant $a_2$ given by \eqref{constants-a-1-a-2-for-scaling-and-mass-conservation}, let 
\begin{equation*}
R=R(t)=\left(\frac{t}{a_2}\right)^{a_2},\qquad \left(t>0\right).
\end{equation*}
For any $M>0$, let
\begin{equation*}
	\widetilde{\mathcal{B}}_M\left(\eta\right)=\left(C-\frac{(p-2)}{p}|\eta|^{\frac{p}{p-1}}\right)_+^{\frac{p-1}{p-2}}.
\end{equation*}
Then
\begin{equation}\label{relatio9n-between-Barenblatt-and-rescaled-Barenblatt-solutions}
	\mathcal{B}_M(x,t)=\left(\frac{t}{a_2}\right)^{-a_1}\widetilde{\mathcal{B}}_M\left(\left(\frac{t}{a_2}\right)^{-a_2}x\right)=\frac{1}{R^n}\widetilde{\mathcal{B}}_M\left(\eta\right)
\end{equation}
where the constant $C$ is uniquely determined by the constants $p$, $n$ and initial mass $M$.\\

\indent For a solution $\bold{u}=\left(u^1,\cdots,u^k\right)$ of \eqref{eq-main-equation-of-system}, let 
\begin{equation*}
	M_l=\int_{\R^n}u^l_0\,dx \qquad \mbox{and} \qquad \bold{M}=\left(M_1,\cdots,M_k\right).
\end{equation*}
For $\widehat{j}\in\N$, let $\bold{u}_{\,\widehat{j}}$ be the smooth solution of \eqref{eq-for-existence-non-degenerate-pde-with-bold-w-epsilon-M-diffusion-coefficients} satisfying 
\begin{equation}\label{eq-L-1-mass-of-u-widehat-j-01}
	M_l=\int_{\R^n}\left(u_{\,\widehat{j}}^l\right)_0\,dx, \qquad \forall 1\leq l\leq k.
\end{equation}
Consider the continuous rescaling
\begin{equation}\label{eq-continuous-rescaling-of=solution-u-i-and-bold-u}
	\theta^{\,l}\left(\eta,\tau\right)=R^{\,n}u^l\left(x,t\right), \qquad \qquad \mbox{\boldmath$\theta$}\left(\eta,\tau\right)=R^{\,n}\bold{u}\left(x,t\right)
\end{equation}
and
\begin{equation*}\label{eq-continuous-rescaling-of=smooth-solution-u-i-widehat-j-and-bold-u}
	\theta_{\,\widehat{j}}^{\,l}\left(\eta,\tau\right)=R^{\,n}u_{\,\widehat{j}}^l\left(x,t\right), \qquad \qquad \mbox{\boldmath$\theta$}_{\,\widehat{j}}\left(\eta,\tau\right)=R^{\,n}\bold{u}_{\,\widehat{j}}\left(x,t\right) \qquad\left(\,\eta=\frac{x}{R},\,\,\,\tau=\log R,\,\,\,1\leq l\leq k\,\right).
\end{equation*}
Then $\mbox{\boldmath$\theta$}_{\,\widehat{j}}=\left(\theta_{\,\widehat{j}}^{\,1},\cdots,\theta_{\,\widehat{j}}^{\,k}\right)$ satisfies
\begin{equation}\label{eq-equation-for-soution-boldmath-theta}
	\begin{aligned}
	\left(\theta_{\,\widehat{j}}^{\,l}\right)_{\tau}&=\nabla\cdot\left(\Theta_{\,\widehat{j}}^{p-2}\nabla\theta_{\,\widehat{j}}^{\,l}\right)+\nabla\theta_{\,\widehat{j}}^{\,l}\cdot\eta+n\theta_{\,\widehat{j}}^{\,l}+\frac{1}{\widehat{j}\,e^{\left(\frac{2a_2-1}{a_2}\right)\tau}}\,\La\theta_{\,\widehat{j}}^l\\
	&=\nabla\cdot\left(\Theta_{\,\widehat{j}}^{p-2}\nabla\theta_{\,\widehat{j}}^{\,l}\right)+\nabla\left(\eta\,\theta_{\,\widehat{j}}^{\,l}\right)+\frac{1}{\widehat{j}\,e^{\left(\frac{2a_2-1}{a_2}\right)\tau}}\,\La\theta_{\,\widehat{j}}^l \qquad \qquad \left(\Theta_{\,\widehat{j}}=\left|\nabla\mbox{\boldmath$\theta$}_{\,\widehat{j}}\right|,\,\,\,1\leq l\leq k\right)
	\end{aligned}
\end{equation}
and
\begin{equation*}
\begin{aligned}
\left(\left|\mbox{\boldmath$\theta$}_{\,\widehat{j}}\right|\right)_{\tau}&=\nabla\cdot\left(\Theta_{\,\widehat{j}}^{p-2}\nabla \left|\mbox{\boldmath$\theta$}_{\,\widehat{j}}\right|\right)+\nabla\cdot\left(\eta\left|\mbox{\boldmath$\theta$}_{\,\widehat{j}}\right|\right)+\frac{\Theta_{\,\widehat{j}}^{p-2}\left|\nabla \left|\mbox{\boldmath$\theta$}_{\,\widehat{j}}\right|\right|^2}{\left|\mbox{\boldmath$\theta$}_{\,\widehat{j}}\right|}-\frac{\Theta_{\,\widehat{j}}^p}{\left|\mbox{\boldmath$\theta$}_{\,\widehat{j}}\right|}\\
&\qquad \qquad \qquad  +\frac{1}{\widehat{j}\,e^{\left(\frac{2a_2-1}{a_2}\right)\tau}}\,\La\left|\mbox{\boldmath$\theta$}_{\,\widehat{j}}\right|+\frac{1}{\widehat{j}\,e^{\left(\frac{2a_2-1}{a_2}\right)\tau}}\left(\frac{\left|\nabla \left|\mbox{\boldmath$\theta$}_{\,\widehat{j}}\right|\right|^2}{\left|\mbox{\boldmath$\theta$}_{\,\widehat{j}}\right|}-\frac{\Theta_{\,\widehat{j}}^2}{\left|\mbox{\boldmath$\theta$}_{\,\widehat{j}}\right|}\right).
\end{aligned}
\end{equation*}
Note that, by Cauchy-Schwarz inequality 
\begin{equation*}
	\left|\nabla \left|\mbox{\boldmath$\theta$}_{\,\widehat{j}}\right|\right|\leq \Theta_{\,\widehat{j}}.
\end{equation*}

For a positive function $f$, we now consider the convex functionals
\begin{equation*}\label{eq-definition-of-convex-functional-H-theta-l-tau}
H_{f}\left(\tau\right)=\int_{\R^n}\left[\sigma\left(f\right)-\sigma\left(\widetilde{\mathcal{B}}_{\left|\bold{M}\right|}\right)-\sigma'\left(\widetilde{\mathcal{B}}_{\left|\bold{M}\right|}\right)\left(f-\widetilde{\mathcal{B}}_{\left|\bold{M}\right|}\right)\right]\,d\eta
\end{equation*}
and
\begin{equation*}\label{eq-definition-of-convex-functional-widehat-H-theta-l-tau-general}
	\widehat{H}_{f}\left(\tau\right)=\int_{\R^n}\left[\sigma\left(f\right)-\sigma\left(\widetilde{\mathcal{B}}_{\left|\bold{M}\right|}\right)+\frac{(p-1)}{p}\left|\eta\right|^{\frac{p}{p-1}}\left(f-\widetilde{\mathcal{B}}_{\left|\bold{M}\right|}\right)\right]\,d\eta
\end{equation*}
where $\sigma:\R^+\to\R$ is a function defined by
\begin{equation*}
\sigma(s)=\frac{(p-1)^2}{(2p-3)(p-2)}s^{\frac{2p-3}{p-1}}.
\end{equation*}
By the second-order Taylor expansion of $\sigma$ around $1$, we can have
\begin{equation}\label{eq-second-order-Taylor-expansion-for-sigma-around-1}
\begin{aligned}
&\sigma\left(\left|\mbox{\boldmath$\theta$}_{\,\widehat{j}}\right|\right)-\sigma\left(\widetilde{\mathcal{B}}_{\left|\bold{M}\right|}\right)-\sigma'\left(\widetilde{\mathcal{B}}_{\left|\bold{M}\right|}\right)\left(\left|\mbox{\boldmath$\theta$}_{\,\widehat{j}}\right|-\widetilde{\mathcal{B}}_{\left|\bold{M}\right|}\right)\\
&\qquad  =\left[\sigma\left(\frac{\left|\mbox{\boldmath$\theta$}_{\,\widehat{j}}\right|}{\widetilde{\mathcal{B}}_{\left|\bold{M}\right|}}\right)-\sigma(1)-\sigma'(1)\left(\frac{\left|\mbox{\boldmath$\theta$}_{\,\widehat{j}}\right|}{\widetilde{\mathcal{B}}_{\left|\bold{M}\right|}}-1\right)\right] \widetilde{\mathcal{B}}_{\left|\bold{M}\right|}^{\frac{2p-3}{p-1}} =\frac{1}{2}\widetilde{\mathcal{B}}_{\left|\bold{M}\right|}^{\,\,-\frac{1}{p-1}}\left|\left|\mbox{\boldmath$\theta$}_{\,\widehat{j}}\right|-\widetilde{\mathcal{B}}_{\left|\bold{M}\right|}\right|^2\sigma''\left(1+\beta\left(\frac{\left|\mbox{\boldmath$\theta$}_{\,\widehat{j}}\right|}{\widetilde{\mathcal{B}}_{\left|\bold{M}\right|}}-1\right)\right)
\end{aligned}
\end{equation}
for some $\beta\in(0,1)$. Thus,
\begin{align}
H_{\left|\mbox{\boldmath$\theta$}_{\,\widehat{j}}\right|}\left(\tau\right)&\geq \int_{\left\{\widetilde{\mathcal{B}}_{\left|\bold{M}\right|}\neq 0\right\}}\left[\sigma\left(\left|\mbox{\boldmath$\theta$}_{\,\widehat{j}}\right|\right)-\sigma\left(\widetilde{\mathcal{B}}_{\left|\bold{M}\right|}\right)-\sigma'\left(\widetilde{\mathcal{B}}_{\left|\bold{M}\right|}\right)\left(\left|\mbox{\boldmath$\theta$}_{\,\widehat{j}}\right|-\widetilde{\mathcal{B}}_{\left|\bold{M}\right|}\right)\right]\,d\eta\notag\\
&=\frac{1}{2}\int_{\left\{\widetilde{\mathcal{B}}_{\left|\bold{M}\right|}\neq 0\right\}}\widetilde{\mathcal{B}}_{\left|\bold{M}\right|}^{\,\,-\frac{1}{p-1}}\left|\left|\mbox{\boldmath$\theta$}_{\,\widehat{j}}\right|-\widetilde{\mathcal{B}}_{\left|\bold{M}\right|}\right|^2\left(1+\beta\left(\frac{\left|\mbox{\boldmath$\theta$}_{\,\widehat{j}}\right|}{\widetilde{\mathcal{B}}_{\left|\bold{M}\right|}}-1\right)\right)^{-\frac{1}{p-1}}\,d\eta\notag\\
&\geq \frac{1}{2}\int_{\left\{\left|\mbox{\boldmath$\theta$}_{\,\widehat{j}}\right|<\widetilde{\mathcal{B}}_{\left|\bold{M}\right|}\right\}}\widetilde{\mathcal{B}}_{\left|\bold{M}\right|}^{\,\,-\frac{1}{p-1}}\left|\left|\mbox{\boldmath$\theta$}_{\,\widehat{j}}\right|-\widetilde{\mathcal{B}}_{\left|\bold{M}\right|}\right|^2\,d\eta\label{eq-lower-bound-of-functional-H-sub-absolute-bold-theta-by-weighted-L-2-norm-of-defferences-between-bold-theta-and-fundamental-sol}\\
&\geq 0.\label{eq-positivity-of-functional-H-sub-absolute-bold-theta}
\end{align}
Since
\begin{equation}\label{computaion-of-sigma-dash-at-Barrenblatt-solution}
\sigma'\left(\widetilde{\mathcal{B}}_{\left|\bold{M}\right|}\right)=\left(\frac{p-1}{p-2}\right)\left(C_{\left|\bold{M}\right|}-\frac{(p-2)}{p}|\eta|^{\frac{p}{p-1}}\right)_+\qquad \mbox{and} \qquad 	\int_{\R^n}\left|\mbox{\boldmath$\theta$}_{\,\widehat{j}}\right|\left(\eta,\tau\right)\,d\eta=\int_{\R^n}\widetilde{\mathcal{B}}_{\left|\bold{M}\right|}\left(\eta\right)\,d\eta,
\end{equation}
we also have
\begin{align}
H_{\left|\mbox{\boldmath$\theta$}_{\,\widehat{j}}\right|}\left(\tau\right)&=\int_{\R^n}\left[\sigma\left(\left|\mbox{\boldmath$\theta$}_{\,\widehat{j}}\right|\right)-\sigma\left(\widetilde{\mathcal{B}}_{\left|\bold{M}\right|}\right)+\frac{(p-1)}{p}\left|\eta\right|^{\frac{p}{p-1}}\left(\left|\mbox{\boldmath$\theta$}_{\,\widehat{j}}\right|-\widetilde{\mathcal{B}}_{\left|\bold{M}\right|}\right)\right]\,d\eta\notag\\
&\qquad +\int_{\left\{\widetilde{\mathcal{B}}_{\left|\bold{M}\right|}=0\right\}}\left(\frac{p-1}{p-2}\right)\left(C_{\left|\bold{M}\right|}-\frac{(p-2)}{p}|\eta|^{\frac{p}{p-1}}\right)\left|\mbox{\boldmath$\theta$}_{\,\widehat{j}}\right|\,d\eta\notag\\
&\leq \int_{\R^n}\left[\sigma\left(\left|\mbox{\boldmath$\theta$}_{\,\widehat{j}}\right|\right)-\sigma\left(\widetilde{\mathcal{B}}_{\left|\bold{M}\right|}\right)+\frac{(p-1)}{p}\left|\eta\right|^{\frac{p}{p-1}}\left(\left|\mbox{\boldmath$\theta$}_{\,\widehat{j}}\right|-\widetilde{\mathcal{B}}_{\left|\bold{M}\right|}\right)\right]\,d\eta\notag\\
&=\widehat{H}_{\left|\mbox{\boldmath$\theta$}_{\,\widehat{j}}\right|}\left(\tau\right).\label{eq-definition-of-convex-functional-widehat-H-theta-l-tau}
\end{align}

We first get the following variation of entropy $\widehat{H}_{\left|\mbox{\boldmath$\theta$}\right|}$.
\begin{lemma}\label{eq-inequality-of-ODE-of-widehat-H-w-r-t-tau}
	Let $n\geq 2$ and $p>2$. Let $\bold{\theta}=\left(\theta^1,\cdots,\theta^k\right)$ be a solution of \eqref{eq-equation-for-soution-boldmath-theta}. If 
	\begin{equation}\label{eq-initial-condition-of-convergence-in-entropy-method}
 \int_{\R^n}\left(1+\left|\bold{u}_0\right|^{\frac{p-2}{p-1}}+\left|x\right|^{\frac{p}{p-1}}\right)\left|\bold{u}_0\right|\,dx<\infty,
	\end{equation}
	then
	\begin{equation}\label{decay-rate-of-widehat-H-w-r-t-time-tau}
    \widehat{H}_{\left|\mbox{\boldmath$\theta$}\right|}\left(\tau\right)\leq e^{-\tau}\cdot \widehat{H}_{\left|\mbox{\boldmath$\theta$}\right|}(0) \qquad \forall \tau>0.
	\end{equation}
\end{lemma}
\begin{proof}
	We will use a modification of the proofs of Theorem 2.2 of \cite{A} and Proposition 1 of \cite{DD} to prove the lemma. Let
	 \begin{equation*}
		\widehat{H}_{\,\widehat{j}}\left(\tau\right)=\widehat{H}_{\left|\mbox{\boldmath$\theta$}_{\,\widehat{j}}\right|}\left(\tau\right).
		\end{equation*}
	By integration by parts and Young's inequality, we have	
	\begin{align}
	\frac{d\widehat{H}_{\,\widehat{j}}}{d\tau}&=\int_{\R^n}\left(\sigma'\left(\left|\mbox{\boldmath$\theta$}_{\,\widehat{j}}\right|\right)+\frac{(p-1)}{p}\left|\eta\right|^{\frac{p}{p-1}}\right)\left(\left|\mbox{\boldmath$\theta$}_{\,\widehat{j}}\right|\right)_{\tau}d\eta\notag\\
	&=-\int_{\R^n}\nabla\left(\sigma'\left(\left|\mbox{\boldmath$\theta$}_{\,\widehat{j}}\right|\right)+\frac{(p-1)}{p}\left|\eta\right|^{\frac{p}{p-1}}\right)\cdot\left(\Theta_{\,\widehat{j}}^{p-2}\nabla\left|\mbox{\boldmath$\theta$}_{\,\widehat{j}}\right|+\eta \left|\mbox{\boldmath$\theta$}_{\,\widehat{j}}\right|\right)d\eta\notag\\
	&\qquad -\int_{\R^n}\frac{\Theta_{\,\widehat{j}}^{p-2}}{\left|\mbox{\boldmath$\theta$}_{\,\widehat{j}}\right|}\left(\sigma'\left(\left|\mbox{\boldmath$\theta$}_{\,\widehat{j}}\right|\right)+\frac{(p-1)}{p}\left|\eta\right|^{\frac{p}{p-1}}\right)\left(\Theta_{\,\widehat{j}}^2-\left|\nabla\left|\mbox{\boldmath$\theta$}_{\,\widehat{j}}\right|\right|^2\right)\,d\eta\notag\\
	&\qquad\qquad -\frac{1}{\widehat{j}\,e^{\left(\frac{2a_2-1}{a_2}\right)\tau}}\int_{\R^n}\nabla\left(\sigma'\left(\left|\mbox{\boldmath$\theta$}_{\,\widehat{j}}\right|\right)+\frac{(p-1)}{p}\left|\eta\right|^{\frac{p}{p-1}}\right)\cdot\nabla\left|\mbox{\boldmath$\theta$}_{\,\widehat{j}}\right|d\eta\notag\\
	&\qquad \qquad \qquad  -\frac{1}{\widehat{j}\,e^{\left(\frac{2a_2-1}{a_2}\right)\tau}}\int_{\R^n}\frac{1}{\left|\mbox{\boldmath$\theta$}_{\,\widehat{j}}\right|}\left(\sigma'\left(\left|\mbox{\boldmath$\theta$}_{\,\widehat{j}}\right|\right)+\frac{(p-1)}{p}\left|\eta\right|^{\frac{p}{p-1}}\right)\left(\Theta_{\,\widehat{j}}^2-\left|\nabla\left|\mbox{\boldmath$\theta$}_{\,\widehat{j}}\right|\right|^2\right)\,d\eta\notag\\
	&\leq -\int_{\R^n}\left(\left|\mbox{\boldmath$\theta$}_{\,\widehat{j}}\right|^{-\frac{1}{p-1}}\nabla\left|\mbox{\boldmath$\theta$}_{\,\widehat{j}}\right|+\left|\eta\right|^{\frac{p}{p-1}-2}\eta\right)\cdot\left(\Theta_{\,\widehat{j}}^{p-2}\nabla\left|\mbox{\boldmath$\theta$}_{\,\widehat{j}}\right|+\eta \left|\mbox{\boldmath$\theta$}_{\,\widehat{j}}\right|\right)\,d\eta\notag\\
	&\qquad -\frac{p-1}{p-2}\int_{\R^n}\left|\mbox{\boldmath$\theta$}_{\,\widehat{j}}\right|^{-\frac{1}{p-1}}\Theta_{\,\widehat{j}}^{p-2}\left(\Theta_{\,\widehat{j}}^2-\left|\nabla\left|\mbox{\boldmath$\theta$}_{\,\widehat{j}}\right|\right|^2\right)\,d\eta  -\frac{p-1}{p}\int_{\R^n}\frac{\left|\eta\right|^{\frac{p}{p-1}}\Theta_{\,\widehat{j}}^{p-2}}{\left|\mbox{\boldmath$\theta$}_{\,\widehat{j}}\right|}\left(\Theta_{\,\widehat{j}}^2-\left|\nabla\left|\mbox{\boldmath$\theta$}_{\,\widehat{j}}\right|\right|^2\right)\,d\eta\notag\\
	&\qquad \qquad+\frac{1}{\widehat{j}\,e^{\left(2-\frac{1}{a_2}\right)\tau}} \int_{\R^n}\left|\eta\right|^{\frac{1}{p-1}}\left|\nabla\left|\mbox{\boldmath$\theta$}_{\,\widehat{j}}\right|\right|\,d\eta\notag\\
	&=-\left(I_1+I_2+I_3+I_4\right)\notag\\
	&\qquad  -\frac{1}{p-2}\int_{\R^n}\left|\mbox{\boldmath$\theta$}_{\,\widehat{j}}\right|^{-\frac{1}{p-1}}\Theta_{\,\widehat{j}}^{p-2}\left(\Theta_{\,\widehat{j}}^2-\left|\nabla\left|\mbox{\boldmath$\theta$}_{\,\widehat{j}}\right|\right|^2\right)\,d\eta-\frac{p-1}{p}\int_{\R^n}\frac{\left|\eta\right|^{\frac{p}{p-1}}\Theta_{\,\widehat{j}}^{p-2}}{\left|\mbox{\boldmath$\theta$}_{\,\widehat{j}}\right|}\left(\Theta_{\,\widehat{j}}^2-\left|\nabla\left|\mbox{\boldmath$\theta$}_{\,\widehat{j}}\right|\right|^2\right)\,d\eta \label{eqw-separate-of-time-differentiable-of-widehat-H-1}\\
	&\qquad \qquad+\frac{1}{\widehat{j}\,e^{\left(2-\frac{1}{a_2}\right)\tau}} \int_{\R^n}\left|\eta\right|^{\frac{1}{p-1}}\left|\nabla\left|\mbox{\boldmath$\theta$}_{\,\widehat{j}}\right|\right|\,d\eta\notag
	\end{align}
	where
	\begin{equation*}
	I_1=\int_{\R^n}\left|\mbox{\boldmath$\theta$}_{\,\widehat{j}}\right|^{-\frac{1}{p-1}}\Theta_{\,\widehat{j}}^{p}\,d\eta, \qquad I_2=\int_{\R^n}\left|\eta\right|^{\frac{p}{p-1}}\left|\mbox{\boldmath$\theta$}_{\,\widehat{j}}\right|\,d\eta
	\end{equation*}
	and
	\begin{equation*}
	I_3=\frac{p-1}{p-2}\int_{\R^n}\left|\mbox{\boldmath$\theta$}_{\,\widehat{j}}\right|\nabla\left|\mbox{\boldmath$\theta$}_{\,\widehat{j}}\right|^{\frac{p-2}{p-1}}\cdot\eta\,d\eta, \qquad I_4=\int_{\R^n}\left|\eta\right|^{\frac{-(p-2)}{p-1}}\Theta_{\,\widehat{j}}^{p-2}\nabla\left|\mbox{\boldmath$\theta$}_{\,\widehat{j}}\right|\cdot\eta\,d\eta.
	\end{equation*}
By H\"older inequality and condition \eqref{eq-initial-condition-of-convergence-in-entropy-method}, there exists a constant $C_1>0$ such that
\begin{equation}\label{eq-upper-bound-of-the-last-term-with-widehat-j=-and-tau}
	\begin{aligned}
	&\frac{1}{\widehat{j}\,e^{\left(2-\frac{1}{a_2}\right)\tau}} \int_{\R^n}\left|\eta\right|^{\frac{1}{p-1}}\left|\nabla\left|\mbox{\boldmath$\theta$}_{\,\widehat{j}}\right|\right|\,d\eta\\
	&\qquad \qquad \leq \frac{1}{\sqrt{\widehat{j}}\,\,e^{\left(2-\frac{1}{a_2}\right)\tau}} \left(\int_{\R^n}\left|\eta\right|^{\frac{p}{p-1}}\left|\mbox{\boldmath$\theta$}_{\,\widehat{j}}\right|\,d\eta\right)^{\frac{1}{p}}\left(\int_{\R^n}\left|\mbox{\boldmath$\theta$}_{\,\widehat{j}}\right|\,d\eta\right)^{\frac{p-2}{2p}}\left( \frac{4}{\widehat{j}}\int_{\R^n}\left|\nabla\left|\mbox{\boldmath$\theta$}_{\,\widehat{j}}\right|^{\frac{1}{2}}\right|^2\,d\eta\right)^{\frac{1}{2}}\\
	&\qquad \qquad \leq \frac{C_1}{\sqrt{\widehat{j}}\,\,e^{\left(2-\frac{1}{a_2}\right)\tau}}.
	\end{aligned}
\end{equation}
    By Young's inequality, we have
	\begin{equation}\label{eq-separation-of-I-4-into-I-1-and-I-2-by-Youngs-inequality}
	\begin{aligned}
	\left|I_4\right|\leq \frac{p-1}{p}\int_{\R^n}\left|\mbox{\boldmath$\theta$}_{\,\widehat{j}}\right|^{-\frac{1}{p-1}}\Theta_{\,\widehat{j}}^p\,d\eta+\frac{1}{p}\int_{\R^n}\left|\eta\right|^{\frac{p}{p-1}}\left|\mbox{\boldmath$\theta$}_{\,\widehat{j}}\right|\,d\eta\leq \frac{p-1}{p}I_1+\frac{1}{p}I_2.
	\end{aligned}
	\end{equation}
	By \eqref{eqw-separate-of-time-differentiable-of-widehat-H-1}, \eqref{eq-upper-bound-of-the-last-term-with-widehat-j=-and-tau}, \eqref{eq-separation-of-I-4-into-I-1-and-I-2-by-Youngs-inequality} and Cauchy-Schwarz inequality, we have
	\begin{equation}\label{inequality-time-deriviatives-of-H-controlled-by-I-1-I-2-and-I-3}
		\begin{aligned}
	\frac{d\widehat{H}_{\,\widehat{j}}}{d\tau}&\leq -\left(\frac{1}{p}I_1+\frac{p-1}{p}I_2+I_3\right)+\frac{C_1\,\,e^{\left(\frac{1}{a_2}-2\right)\tau}}{\sqrt{\widehat{j}}}\\
	&\leq -\left(\frac{1}{p}I'_1+\frac{p-1}{p}I_2+I_3\right)+\frac{C_1\,\,e^{\left(\frac{1}{a_2}-2\right)\tau}}{\sqrt{\widehat{j}}}
	\end{aligned}
	\end{equation}
	where
	\begin{equation*}
	I'_1=\int_{\R^n}\left|\mbox{\boldmath$\theta$}_{\,\widehat{j}}\right|^{-\frac{1}{p-1}}\left|\nabla \left|\mbox{\boldmath$\theta$}_{\,\widehat{j}}\right|\right|^{p}\,d\eta.
	\end{equation*}
	On the other hand, by the equation (36) in the proof of Theorem 2.2 of \cite{A}, the functional $\widehat{H}_{\,\widehat{j}}$ is also bounded from below by the last term of \eqref{inequality-time-deriviatives-of-H-controlled-by-I-1-I-2-and-I-3}, i.e., 
	\begin{equation}\label{inequality-of-inverse-situation-H-controlled-by-I-1-I-2-and-I-3}
	\widehat{H}_{\,\widehat{j}}\leq \frac{1}{p}I'_1+\frac{p-1}{p}I_2+I_3.
	\end{equation}
	By \eqref{inequality-time-deriviatives-of-H-controlled-by-I-1-I-2-and-I-3} and \eqref{inequality-of-inverse-situation-H-controlled-by-I-1-I-2-and-I-3},
	\begin{equation}\label{inequality-of-ODE-of-widehat-H-elimination-of-I-1-I-2-and-I-3}
	\left(\widehat{H}_{\,\widehat{j}}\right)_{\tau}\leq -\widehat{H}_{\,\widehat{j}}+\frac{C_1\,\,e^{\left(\frac{1}{a_2}-2\right)\tau}}{\sqrt{\widehat{j}}} \qquad \Rightarrow \qquad \left(e^{\tau}\widehat{H}_{\,\widehat{j}}-\frac{C_1}{\left(\frac{1}{a_2}-1\right)\sqrt{\widehat{j}}}\,\,e^{\left(\frac{1}{a_2}-1\right)\tau}\right)_{\tau}\leq 0 \qquad \forall \tau>0.
	\end{equation}
	Integrating \eqref{inequality-of-ODE-of-widehat-H-elimination-of-I-1-I-2-and-I-3} over $\left(0,\tau\right)$, we have 
	\begin{equation*}
		\widehat{H}_{\,\widehat{j}}\left(\tau\right)\leq e^{-\tau}\widehat{H}_{\,\widehat{j}}\left(0\right)+\frac{C_1}{\left(\frac{1}{a_2}-1\right)\sqrt{\widehat{j}}}\,\,\left(e^{\left(\frac{1}{a_2}-1\right)\tau}-1\right).
	\end{equation*}
This immediately implies 
\begin{equation*}
\widehat{H}_{\left|\mbox{\boldmath$\theta$}\right|}\left(\tau\right)	\leq \liminf_{\widehat{j}\to\infty}\widehat{H}_{\,\widehat{j}}\left(\tau\right)\leq \lim_{\widehat{j}\to\infty}\left[e^{-\tau}\widehat{H}_{\,\widehat{j}}\left(0\right)+\frac{C_1}{\left(\frac{1}{a_2}-1\right)\sqrt{\widehat{j}}}\,\,\left(e^{\left(\frac{1}{a_2}-1\right)\tau}-1\right)\right]=e^{-\tau}\cdot \widehat{H}_{\left|\mbox{\boldmath$\theta$}\right|}(0).
\end{equation*}
	Therefore,the inequality \eqref{decay-rate-of-widehat-H-w-r-t-time-tau} and the lemma follows.
\end{proof}

\indent Next, we investigate the similarity between the components of $\mbox{\boldmath$\theta$}$ at infinity.
\begin{lemma}\label{lem-relation-between-component-theta-l-and-absolute-of-bold-theta-at-infty}
	Suppose that the solution $\mbox{\boldmath$\theta$}=\left(\theta^1,\cdots,\theta^k\right)$ converges as $\tau\to\infty$. Then, under the hypotheses of Lemma \ref{eq-inequality-of-ODE-of-widehat-H-w-r-t-tau} we also have
	\begin{equation}\label{eq-relation-between-theta-l-and-absolute-of-bold-theta-at-infty-in-Lemma}
	\lim_{\tau\to\infty}\theta^l=\lim_{\tau\to\infty}\left(\frac{M_l}{{\left|\bold{M}\right|}}\left|\mbox{\boldmath$\theta$}\right|\right) \qquad \mbox{a.e. in $\R^n$}.
	\end{equation}
\end{lemma}

\begin{proof}
	Let  $I_1$, $I_2$, $I_3$, $I_4$ and $I'_1$ be given in the proof of Lemma \ref{eq-inequality-of-ODE-of-widehat-H-w-r-t-tau} and let
		\begin{equation*}
	\lim_{\tau\to\infty}\theta^{\,l}=\widetilde{\theta}^{\,l} \qquad \mbox{and} \qquad \lim_{\tau\to\infty}\mbox{\boldmath$\theta$}=\widetilde{\mbox{\boldmath$\theta$}}.
	\end{equation*} 
	Let
	\begin{equation*}
	H_{\,\widehat{j}}\left(\tau\right)=H_{\left|\mbox{\boldmath$\theta$}_{\,\widehat{j}}\right|}\left(\tau\right) \qquad \mbox{and} \qquad \widehat{H}_{\,\widehat{j}}\left(\tau\right)=\widehat{H}_{\left|\mbox{\boldmath$\theta$}_{\,\widehat{j}}\right|}\left(\tau\right).
	\end{equation*}
	By \eqref{eq-positivity-of-functional-H-sub-absolute-bold-theta} and \eqref{eq-definition-of-convex-functional-widehat-H-theta-l-tau} we have
	\begin{equation}\label{negative-lower-bound-of-functional-widehat-H}
	\widehat{H}_{\,\widehat{j}}(\tau)\geq H_{\,\widehat{j}}\left(\tau\right)\geq 0.
	\end{equation}
	By \eqref{eqw-separate-of-time-differentiable-of-widehat-H-1}, \eqref{eq-upper-bound-of-the-last-term-with-widehat-j=-and-tau}, \eqref{eq-separation-of-I-4-into-I-1-and-I-2-by-Youngs-inequality} and \eqref{inequality-of-inverse-situation-H-controlled-by-I-1-I-2-and-I-3}, we can get
	\begin{equation*}
		e^{\tau}\int_{\R^n}\left(\frac{1}{p-2}\left|\mbox{\boldmath$\theta$}_{\,\widehat{j}}\right|^{\frac{p-2}{p-1}}+\frac{p-1}{p}|\eta|^{\frac{p}{p-1}}\right)\frac{\Theta_{\,\widehat{j}}^{p-2}}{\left|\mbox{\boldmath$\theta$}_{\,\widehat{j}}\right|}\left(\Theta_{\,\widehat{j}}^2-\left|\nabla\left|\mbox{\boldmath$\theta$}_{\,\widehat{j}}\right|\right|^2\right)\,d\eta\leq \left(-e^{\tau}\widehat{H}_{\,\widehat{j}}+\frac{C_1}{\left(\frac{1}{a_2}-1\right)\sqrt{\widehat{j}}}\,\,e^{\left(\frac{1}{a_2}-1\right)\tau}\right)_{\tau}.
    \end{equation*}	
	By \eqref{negative-lower-bound-of-functional-widehat-H}, integrating over $\tau>0$ gives
	\begin{equation*}
		\begin{aligned}
			0\leq&\int_{\frac{\tau}{2}}^{\tau}e^{s}\int_{\R^n}\left(\frac{1}{p-2}\left|\mbox{\boldmath$\theta$}_{\,\widehat{j}}\right|^{\frac{p-2}{p-1}}+\frac{p-1}{p}|\eta|^{\frac{p}{p-1}}\right)\frac{\Theta_{\,\widehat{j}}^{p-2}}{\left|\mbox{\boldmath$\theta$}_{\,\widehat{j}}\right|}\left(\Theta_{\,\widehat{j}}^2-\left|\nabla\left|\mbox{\boldmath$\theta$}_{\,\widehat{j}}\right|\right|^2\right)\,d\eta ds\\
		&\qquad \qquad \leq \int_{0}^{\tau}e^{s}\int_{\R^n}\left(\frac{1}{p-2}\left|\mbox{\boldmath$\theta$}_{\,\widehat{j}}\right|^{\frac{p-2}{p-1}}+\frac{p-1}{p}|\eta|^{\frac{p}{p-1}}\right)\frac{\Theta_{\,\widehat{j}}^{p-2}}{\left|\mbox{\boldmath$\theta$}_{\,\widehat{j}}\right|}\left(\Theta_{\,\widehat{j}}^2-\left|\nabla\left|\mbox{\boldmath$\theta$}_{\,\widehat{j}}\right|\right|^2\right)\,d\eta ds\\
		&\qquad \qquad \leq -e^{\tau}\widehat{H}_{\,\widehat{j}}\left(\tau\right)+\frac{C_1}{\left(\frac{1}{a_2}-1\right)\sqrt{\widehat{j}}}\,\,e^{\left(\frac{1}{a_2}-1\right)\tau}+\widehat{H}_{\,\widehat{j}}\left(0\right)-\frac{C_1}{\left(\frac{1}{a_2}-1\right)\sqrt{\widehat{j}}}\\
		&\qquad \qquad \leq \frac{C_1}{\left(\frac{1}{a_2}-1\right)\sqrt{\widehat{j}}}\,\,e^{\left(\frac{1}{a_2}-1\right)\tau}+\widehat{H}_{\,\widehat{j}}\left(0\right).
		\end{aligned}
	\end{equation*}	
	This immediately implies that
	\begin{equation*}
		\begin{aligned}
			0\leq&\int_{\frac{\tau}{2}}^{\tau}e^{s}\int_{\R^n}\left(\frac{1}{p-2}\left|\mbox{\boldmath$\theta$}\right|^{\frac{p-2}{p-1}}+\frac{p-1}{p}|\eta|^{\frac{p}{p-1}}\right)\frac{\Theta^{p-2}}{\left|\mbox{\boldmath$\theta$}\right|}\left(\Theta^2-\left|\nabla\left|\mbox{\boldmath$\theta$}\right|\right|^2\right)\,d\eta ds\\
			&\qquad \leq \liminf_{\widehat{j}\to\infty}\left[\int_{\frac{\tau}{2}}^{\tau}e^{s}\int_{\R^n}\left(\frac{1}{p-2}\left|\mbox{\boldmath$\theta$}_{\,\widehat{j}}\right|^{\frac{p-2}{p-1}}+\frac{p-1}{p}|\eta|^{\frac{p}{p-1}}\right)\frac{\Theta_{\,\widehat{j}}^{p-2}}{\left|\mbox{\boldmath$\theta$}_{\,\widehat{j}}\right|}\left(\Theta_{\,\widehat{j}}^2-\left|\nabla\left|\mbox{\boldmath$\theta$}_{\,\widehat{j}}\right|\right|^2\right)\,d\eta ds\right]\\
			&\qquad \qquad \leq \lim_{\widehat{j}\to\infty}\left[\frac{C_1}{\left(\frac{1}{a_2}-1\right)\sqrt{\widehat{j}}}\,\,e^{\left(\frac{1}{a_2}-1\right)\tau}+\widehat{H}_{\,\widehat{j}}\left(0\right)\right]\\
			&\qquad \qquad =\widehat{H}_{\left|\mbox{\boldmath$\theta$}\right|}(0).
		\end{aligned}
	\end{equation*}	
	Taking the mean value in $s$ in $\left[\frac{\tau}{2},\tau\right]$, we have
	\begin{equation}\label{inequalities-of-functional-widehat-H-which-showes-negativity-of-time-derivative-of-widehat-H}
		0\leq \int_{\R^n}\left(\frac{1}{p-2}\left|\mbox{\boldmath$\theta$}\right|^{\frac{p-2}{p-1}}+\frac{p-1}{p}|\eta|^{\frac{p}{p-1}}\right)\frac{\Theta^{p-2}}{\left|\mbox{\boldmath$\theta$}\right|}\left(\Theta^2-\left|\nabla\left|\mbox{\boldmath$\theta$}\right|\right|^2\right)\,d\eta\left(\tau'\right)\leq \frac{\tau \widehat{H}_{\left|\mbox{\boldmath$\theta$}\right|}(0)}{2e^{\frac{\tau}{2}}} \qquad \left(\frac{\tau}{2}<\tau'<\tau\right).
	\end{equation}	
	Letting $\tau\to\infty$ in \eqref{inequalities-of-functional-widehat-H-which-showes-negativity-of-time-derivative-of-widehat-H}, we can get
	\begin{equation}\label{eq-equal-of-two-quantity-Theta-and-nabla-absolute-of-bold-theta}
\left|\nabla\,\widetilde{\mbox{\boldmath$\theta$}}\,\right|^2=\left|\nabla\left|\widetilde{\mbox{\boldmath$\theta$}}\right|\,\right|^2.
	\end{equation}
	By \eqref{eq-equal-of-two-quantity-Theta-and-nabla-absolute-of-bold-theta} and the condition of equality in the Cauchy-Schwarz inequality, there exist constants $\bold{c}>0$ and $c^{ij}>0$, $\left(1\leq i,\,j\leq k\right)$, such that
	\begin{equation*}
	\nabla \widetilde{\theta}^{\,i}=c^{ij}\nabla\widetilde{\theta}^{\,j} \qquad \forall 1\leq i,\,j\leq k \qquad \mbox{and} \qquad \widetilde{\theta}^l=\bold{c}\left|\nabla \widetilde{\theta}^l\right| \qquad \forall  1\leq l\leq k.
	\end{equation*}
	This immediately implies that
	\begin{equation}\label{eq-relation-between-theta-l-and-absolute-of-bold-theta-at-infty-in-proof}
	\left|\widetilde{\mbox{\boldmath$\theta$}}\right|^2=\sum_{j=1}^k\left(\,\widetilde{\theta}^{\,j}\,\right)^2=\bold{c}^2\sum_{j=1}^k\left|\nabla\widetilde{\theta}^{\,j}\right|^2=\bold{c}^2\left|\nabla\widetilde{\theta}^{\,i}\right|^2\sum_{j=1}^k\left(c^{ij}\right)^2=\left(\,\widetilde{\theta}^{\,i}\,\right)^2\sum_{j=1}^k\left(c^{ij}\right)^2:=C^l\left(\,\widetilde{\theta}^i\,\right)^2.
	\end{equation}
	By \eqref{eq-relation-between-theta-l-and-absolute-of-bold-theta-at-infty-in-proof} and the $L^1$ mass conservation of $\theta^{\,l}$, one can easily check that the constant $C^l=\frac{{\left|\bold{M}\right|}}{M_l}$. Therefore the equation \eqref{eq-relation-between-theta-l-and-absolute-of-bold-theta-at-infty-in-Lemma} holds and the lemma follows.
\end{proof}

We now ready to show the asymptotic large time behaviour of solution $\bold{u}=\left(u^1,\cdots,u^k\right)$.

\begin{proof}[\textbf{Proof of Theorem \ref{thm-convergence-between-absolute-bold-theta-and-fundamental-function-with-decay-rate}}]
	 By Lemma \ref{lem-relation-between-component-theta-l-and-absolute-of-bold-theta-at-infty}, it suffices to show that 
	\begin{equation}\label{1st-conclusion-of-convergence-between-absolute-bold-theta-and-fundamental-function-with-decay-rate}
	\left\|\,\left|\bold{u}\right|(\cdot,t)-\mathcal{B}_{\left|\bold{M}\right|}(\cdot,t)\right\|_{L^1\left(\R^n\right)}\leq C\left(\frac{t}{a_2}\right)^{-\frac{a_2}{2}}
	\end{equation}
	for some constant $C>0$. Let  $\mbox{\boldmath$\theta$}=\left(\theta^{\,1},\cdots,\theta^{\,k}\right)$ be a solution of \eqref{eq-equation-for-soution-boldmath-theta} given by \eqref{eq-continuous-rescaling-of=solution-u-i-and-bold-u}. Since $\left|\mbox{\boldmath$\theta$}\right|$ and $\widetilde{\mathcal{B}}_M$ have equal mass, 
	\begin{equation}\label{eq-integral-application-of-equal-mass-between-absolute-of-bold-theta-and-fundamental-solution}
	\frac{1}{2}\int_{\R^n}\left|\left|\mbox{\boldmath$\theta$}\right|-\widetilde{\mathcal{B}}_{\left|\bold{M}\right|}\right|\,d\eta=\int_{\left\{\left|\mbox{\boldmath$\theta$}\right|<\widetilde{\mathcal{B}}_M\right\}}\left|\left|\mbox{\boldmath$\theta$}\right|-\widetilde{\mathcal{B}}_{\left|\bold{M}\right|}\right|\,d\eta.
	\end{equation}
	By \eqref{eq-lower-bound-of-functional-H-sub-absolute-bold-theta-by-weighted-L-2-norm-of-defferences-between-bold-theta-and-fundamental-sol}, \eqref{eq-definition-of-convex-functional-widehat-H-theta-l-tau}, \eqref{decay-rate-of-widehat-H-w-r-t-time-tau},  \eqref{eq-integral-application-of-equal-mass-between-absolute-of-bold-theta-and-fundamental-solution} and H\"older inequality,
	\begin{equation}\label{eq-L-1-difference-between-bold-theta-and-widehat-fundamental-solution-in-scaling}
	\begin{aligned}
\int_{\R^n}\left|\left|\mbox{\boldmath$\theta$}\right|-\widetilde{\mathcal{B}}_{\left|\bold{M}\right|}\right|\,d\eta&\leq 2\left(\int_{\left\{\left|\mbox{\boldmath$\theta$}\right|<\widetilde{\mathcal{B}}_{\left|\bold{M}\right|}\right\}}\widetilde{\mathcal{B}}_{\left|\bold{M}\right|}^{\,\,-\frac{1}{p-1}}\left|\left|\mbox{\boldmath$\theta$}\right|-\widetilde{\mathcal{B}}_{\left|\bold{M}\right|}\right|^2\,d\eta\right)^{\frac{1}{2}}\left(\int_{\R^n}\widetilde{\mathcal{B}}_{\left|\bold{M}\right|}^{\,\,\frac{1}{p-1}}\,d\eta\right)^{\frac{1}{2}}\\
	&\leq 2\left(2H_{\left|\mbox{\boldmath$\theta$}\right|}\left(\tau\right)\right)^{\frac{1}{2}}\left(\int_{\R^n}\widetilde{\mathcal{B}}_{\left|\bold{M}\right|}^{\,\,\frac{1}{p-1}}\,d\eta\right)^{\frac{1}{2}}\\
	&\leq 2\left(2e^{-\tau}\widehat{H}_{\left|\mbox{\boldmath$\theta$}\right|}(0)\right)^{\frac{1}{2}}\left(\int_{\R^n}\widetilde{\mathcal{B}}_{\left|\bold{M}\right|}^{\,\,\frac{1}{p-1}}\,d\eta\right)^{\frac{1}{2}}=Ce^{-\frac{\tau}{2}}
	\end{aligned}
	\end{equation}
	where
	\begin{equation*}
	C=2\sqrt{2\widehat{H}_{\left|\mbox{\boldmath$\theta$}\right|}(0)\int_{\R^n}\widetilde{\mathcal{B}}_{\left|\bold{M}\right|}^{\,\,\frac{1}{p-1}}\,d\eta}.
	\end{equation*}
	Therefore, by \eqref{eq-continuous-rescaling-of=solution-u-i-and-bold-u}, \eqref{relatio9n-between-Barenblatt-and-rescaled-Barenblatt-solutions} and \eqref{eq-L-1-difference-between-bold-theta-and-widehat-fundamental-solution-in-scaling} we have
	\begin{equation*}
	\begin{aligned}
	\int_{\R^n}\left|\left|\bold{u}\right|(x,t)-\mathcal{B}_{\left|\bold{M}\right|}\left(x,t\right)\right|\,dx&=\int_{\R^n}\left|\left|\bold{u}\right|(x,t)-\left(\frac{t}{a_2}\right)^{-a_1}\widetilde{\mathcal{B}}_{\left|\bold{M}\right|}\left(\left(\frac{t}{a_2}\right)^{-a_2}x\right)\right|\,dx\\
	&=\int_{\R^n}\left|\left|\mbox{\boldmath$\theta$}\right|(x,t)-\widetilde{\mathcal{B}}_{\left|\bold{M}\right|}\right|\,d\eta\leq C\left(\frac{t}{a_2}\right)^{-\frac{a_2}{2}}.
	\end{aligned}
	\end{equation*}
	Henec the claim holds and the theorem follows.
\end{proof}

\begin{remark}
	By H\"older continuity of component $u^l$, $\left(i=1,\cdots,k\right)$,
	\begin{equation*}
	\theta^{\,i}(\eta,\tau) \to \widetilde{\mathcal{B}}_{\left|\bold{M}\right|}(\eta) \qquad \mbox{on every compact subset of $\R^n$ $\quad$as $\tau\to\infty$}.
	\end{equation*}
   This immediately implies the uniform convergence between $u^{\,i}$ and $\mathcal{B}_M$ in $L_{loc}^{\infty}$, i.e.,
	\begin{equation}\label{eq-L-infty-convergence-between-u-i-and-M-i-barrenblatt-profiles}
	\lim_{t\to\infty}t^{\,a_1}\left|u^l(\cdot,t)-\frac{M_i}{\left|\bold{M}\right|}\mathcal{B}_{\left|\bold{M}\right|}\left(\cdot,t\right)\right|=0\qquad \mbox{uniformly on every compact subset of $\R^n$}.
	\end{equation}
\end{remark}

\section{Harnack Type Inequality of Degenerated $p$-Laplacian System}
\setcounter{equation}{0}
\setcounter{thm}{0}

The last section is devoted to find a suitable Harnack type inequality for the component $u^l$, $\left(1\leq l\leq k\right)$, of solution of \eqref{eq-main-equation-of-system} which makes the size of spatial average of $u^l$ under control by the value of $u^l$ at one point. We will use a modification of techniques in the proof of Theorem 1.5 of \cite{KL2} to show the regularity theory.\\
\indent We start the proof of Harnack type inequality by reviewing a Harnack-type estimate of the $p$-Laplacian equation
\begin{equation}\label{eq-p-Laplacian-equation-in-section-waiting-time}
	u_t=\La_p u=\nabla\cdot\left(\left|\nabla u\right|^{p-2}\nabla u\right). 
\end{equation}
\begin{lemma}[cf. Result 5 of \cite{KV} and Theorem 3.1 of \cite{AC}]\label{lem-harnack-type-estimate-by-A-and-Caffarelli}
	Let $u$ be a nonegative solution of \eqref{eq-p-Laplacian-equation-in-section-waiting-time} in $\R^n\times\left[0,T\right]$ for some $T>0$. Then, for every $R>0$ there exists a constant $C=C\left(m,n\right)>0$ such that
	\begin{equation*}
		\int_{\left\{|x|<R\right\}}u(x,0)\,dx\leq C\left(\frac{R^{n+\frac{p}{p-2}}}{T^{\frac{1}{p-2}}}+T^{\frac{n}{p}}u^{1+\frac{n(p-2)}{p}}\left(0,T\right) \right).
	\end{equation*}
\end{lemma}

Let $\mathcal{B}_M(x,t)$ be the fundamental solution of the $p$-Laplacian equation with $L^1$ mass $M$. Then, by Lemma \ref{lem-harnack-type-estimate-by-A-and-Caffarelli} we can get
\begin{equation}\label{eq-harnack-type-estimates-for-barrenblatt}
	\int_{\left\{|x|<R\right\}}\mathcal{B}_{M}(x,0)\,dx\leq C\left(\frac{R^{n+\frac{p}{p-2}}}{T^{\frac{1}{p-2}}}+T^{\frac{n}{p}}\mathcal{B}_M^{1+\frac{n(p-2)}{p}}\left(0,T\right) \right).
\end{equation}

We are now ready to give a proof for our Harnack type inequality.

\begin{proof}[\textbf{Proof of Theorem \ref{lem-calculation-about-waiting-time-lower-bound}}]
	The proof is almost same as the proof for Theorem 1.5 of \cite{KL2}. For the future references, we will give the sketch of proof here.\\
	\indent For any $M>0$, denote by $P\left(M\right)$ the class of all non-negative weak solution $\overline{\bold{u}}=\left(\overline{u}^1,\cdots,\overline{u}^k\right)$ of 
	\begin{equation*}
		\left(\overline{u}^l\right)_t=\nabla\cdot\left(\left|\nabla\overline{\bold{u}}\right|^{p-2}\nabla \overline{u}^l\right) \qquad \mbox{in $\R^n\times\left[0,\infty\right)$}
	\end{equation*}
	satisfying
	\begin{equation*}
		\sup_{t>0}\int_{\R^n}\left|\overline{\bold{u}}\right|(x,t)\,dx\leq M.
	\end{equation*} 
Let $k\in\N$ and $T>0$ be fixed. By Lemma \ref{lem-Law-of-L-1-Mass-Conservation}, there exists a constant $M_{\bold{u}}>0$, depending on $\bold{u}\left(\cdot,0\right)$, such that
	\begin{equation}\label{eq-bound-of-absolute-of-bold-u-in-L-1-for-all-time-sup}
		\int_{\R^n}\left|\bold{u}\right|(x,t)\,dx\leq M_{\bold{u}} \qquad \forall 0\leq t\leq T.
	\end{equation}
	Thus, we are going to show that \eqref{eq-claim-L-1-norm=of-u-at-initial-is-bounded-by-H-sub-m} holds when $\bold{u}\in P\left(M\right)$ for some constant $M>0$. We divide the proof into two cases.\\
	\textbf{Case 1.} $\textbf{supp} \left|\bold{u}\right|(\cdot,0)\subset B_1=\left\{x\in\R^n:|x|\leq 1\right\}$.\\
	\indent If \eqref{eq-claim-L-1-norm=of-u-at-initial-is-bounded-by-H-sub-m} is violated, then for each $j\in\N$ there exist a constant $R_j>0$, a solution $\bold{u}_j=\left(u^1_j,\cdots,u^{k}_j\right)\in P\left(M_j\right)$ and a number $1\leq i'(j)\leq k$ such that 
	\begin{equation}\label{eq-asumption-if-eq-claim-L-1-norm=of-u-at-initial-is-bounded-by-H-sub-m-violated}
		\int_{\left\{|x|<R_j\right\}}u^{i\,'}_j(x,0)\,dx\geq\frac{j}{\left(\mu^{i\,'}_j\right)^{1+\frac{n(p-2)}{p}}}\left(\frac{R_j^{n+\frac{p}{p-2}}}{T^{\frac{1}{p-2}}}+T^{\frac{n}{p}}\left(u_j^{i'}\right)^{1+\frac{n(p-2)}{p}}\left(0,T\right)\right)
	\end{equation}
	where
	\begin{equation*}
		\mu_j^{\,i\,'}=\frac{\int_{\R^n}u_j^{\,i\,'}(x,0)\,dx}{\max_{1\leq l\leq k}\left\{\int_{\R^n}u_j^l(x,0)\,dx\right\}} \qquad \forall 1\leq i\leq k.
	\end{equation*}
	Without loss of generality we may assume that $i\,'=1$ for each $j\in\N$ and let
	\begin{equation}\label{eq-definition-of-I-i-L-1-mass-of-u-i-max}
		I_j=\max_{1\leq l\leq k}\left\{\int_{\R^n}u^{l}_j(x,0)\,dx\right\}.
	\end{equation}
	By \eqref{eq-asumption-if-eq-claim-L-1-norm=of-u-at-initial-is-bounded-by-H-sub-m-violated} and \eqref{eq-definition-of-I-i-L-1-mass-of-u-i-max},
\begin{equation*}
	I_j\to\infty \qquad \mbox{ as $j\to\infty$}
\end{equation*} 
since $R_j^{n+\frac{p}{p-2}}T^{-\frac{1}{p-2}}\geq T^{\frac{n}{p}}>0$. Consider the rescaled function 
	\begin{equation*}
		v^l_j(x,t)=\frac{1}{I_j^{\frac{p}{n(p-2)+p}}}u^l_j\left(I_j^{\frac{p-2}{n(p-2)+p}}x,t\right) \qquad \forall 1\leq l\leq k.
	\end{equation*}
	By direct computation, one can easily check that $\bold{v}_j=\left(v_j^1,\cdots,v^{k}_j\right)$ is also a solution of \eqref{eq-system-PME-for-waiting-time-1} in $\R^n\times[0,\infty)$ with 
	\begin{equation}\label{eq-L-1-mass-of-rescaled-function-v-i-s}
		0<\mu^0\leq \mu_j^l=\int_{\R^n}v^{i}_{j}(x,0)\,dx\leq 1 \qquad \forall 1\leq l\leq k,\,j\in\N.
	\end{equation}  
Moreover,
	\begin{equation*}
		\textbf{supp}\,v_j^{\,i}(\cdot, 0)\subset B_{\frac{1}{I_j^{\,\,\frac{p-2}{n(p-2)+p}}}}\left(0\right) \qquad \forall 1\leq l\leq k,\,\,j\in\N.
	\end{equation*}
	By the Ascoli theorem and a diagonalization argument the sequence $\left\{\bold{v}_j\right\}_{j=1}^{\infty}$ has a subsequence which we may assume without loss of generality to be the sequence itself such that
	\begin{equation}\label{eq-convergence-opf-int-v-j-i-to-mass-mu-i-1}
		v_j^{\,i}(x,0)\to \mu^{\,i}\delta \qquad \forall 1\leq l\leq k
	\end{equation}
	where $\delta$ is Dirac's delta function and $\mu^{\,i}$ is a constant such that $0<\mu^0\leq \mu^{\,i}\leq 1$. Then, by the Theorem \ref{thm-convergence-between-absolute-bold-theta-and-fundamental-function-with-decay-rate},
	\begin{equation}\label{eq-convergence-of-scaled-v-j-to-Barrenblatt-with-mass-absoluiton-bold-M}
		v_j^l\to \frac{\mu^{\,i}}{\mbox{\boldmath$\mu$}}\mathcal{B}_{\mbox{\boldmath$\mu$}}\qquad \mbox{uniformly on compact subset of $\R^n\times(0,\infty)$}
	\end{equation}
	where $\mbox{\boldmath$\mu$}=\sqrt{\sum_{i=1}^k\left(\mu^{\,i}\right)^2}$ and $\mathcal{B}_{\mbox{\boldmath$\mu$}}$ is the fundamental solution of $p$-Laplacian equation with $L^1$ mass $\mbox{\boldmath$\mu$}$. By \eqref{eq-L-1-mass-of-rescaled-function-v-i-s}, \eqref{eq-convergence-opf-int-v-j-i-to-mass-mu-i-1} and \eqref{eq-convergence-of-scaled-v-j-to-Barrenblatt-with-mass-absoluiton-bold-M}, there exists a number $j_0\in\N$ such that
	\begin{equation}\label{eq-comparison-between-mu-i-j-and-mu-i-1}
		\mu^1_j=\mu^{i\,'}_j\geq \frac{1}{2}\mu^{i\,'}=\frac{1}{2}\mu^1\qquad \forall j\geq j_0
	\end{equation}
	and
	\begin{equation}\label{eq-compare-between-v-j-1-and-barrenblat-with-M-1-by-converging-1}
		\int_{\left\{|x|<I_j^{-\frac{p-2}{n{p-2}+p}}R_j\right\}}v^1_j(x,0)\,dx\leq \mu_j^1\leq 2\mu^1\leq 2\int_{\left\{|x|<I_j^{-\frac{p-2}{n(p-2)+p}}R_j\right\}}\mathcal{B}_{\mbox{\boldmath$\mu$}}(x,0)\,dx  \qquad \forall j\geq j_0
	\end{equation}
	and
	\begin{equation}\label{eq-compare-between-v-j-1-and-barrenblat-with-M-1-by-converging-2}
		\frac{\mu^1}{\mbox{\boldmath$\mu$}}\mathcal{B}_{\mbox{\boldmath$\mu$}}(0,T)\leq 2v_j^1\left(0,T\right)=\frac{2}{I_j^{\frac{p}{n(p-2)+p}}}u^1_j(0,T) \qquad \forall j\geq j_0.
	\end{equation}
	By \eqref{eq-harnack-type-estimates-for-barrenblatt}, \eqref{eq-comparison-between-mu-i-j-and-mu-i-1}, \eqref{eq-compare-between-v-j-1-and-barrenblat-with-M-1-by-converging-1} and \eqref{eq-compare-between-v-j-1-and-barrenblat-with-M-1-by-converging-2}
	\begin{equation*}
		\begin{aligned}
			\frac{1}{I_j}\int_{\left\{|x|<R_j\right\}}u^1_j(x,0)\,dx&=\int_{\left\{|x|<I_j^{-\frac{p-2}{n(p-2)+p}}R_j\right\}}v^1_j(x,0)\,dx\\
			&\leq 2\int_{\left\{|x|<I_j^{-\frac{p-2}{n(p-2)+p}}R_j\right\}}\mathcal{B}_{\mbox{\boldmath$\mu$}}(x,0)\,dx\leq 2C\left(\frac{R_j^{n+\frac{p}{p-2}}}{I_j\,T^{\frac{1}{p-2}}}+T^{\frac{n}{p}}\mathcal{B}^{1+\frac{n(p-2)}{p}}_{\mbox{\boldmath$\mu$}}(0,T)\right)\\
			&\leq\frac{2^{2+\frac{n(p-2)}{p}}k^{\frac{1}{2}+\frac{n(p-2)}{2p}}C}{\left(\mu^1\right)^{1+\frac{n(p-2)}{p}}I_j}\left(\frac{R_j^{n+\frac{p}{p-2}}}{T^{\frac{1}{p-2}}}+T^{\frac{n}{p}}\left(u^1_j\right)^{1+\frac{n(p-2)}{p}}(0,T)\right)\\
			&\leq\frac{2^{3+\frac{2n(p-2)}{p}}k^{\frac{1}{2}+\frac{n(p-2)}{2p}}C}{\left(\mu_j^1\right)^{1+\frac{n(p-2)}{p}}I_j}\left(\frac{R_j^{n+\frac{p}{p-2}}}{T^{\frac{1}{p-2}}}+T^{\frac{n}{p}}\left(u^1_j\right)^{1+\frac{n(p-2)}{p}}(0,T)\right) \qquad \forall j\geq j_0.
		\end{aligned}
	\end{equation*}
	Hence, for $C_1=2^{3+\frac{2n(p-2)}{p}}k^{\frac{1}{2}+\frac{n(p-2)}{2p}}C$
	\begin{equation*}
		\begin{aligned}
			\int_{\left\{|x|<R_j\right\}}u^1_j(x,0)\,dx&\leq \frac{C_1}{\left(\mu^1_j\right)^{1+\frac{n(p-2)}{p}}}\left(\frac{R_j^{n+\frac{p}{p-2}}}{T^{\frac{1}{p-2}}}+T^{\frac{n}{p}}\left(u^1_j\right)^{1+\frac{n(p-2)}{p}}(0,T)\right)  \qquad \forall j\geq j_0
		\end{aligned}
	\end{equation*}
	, which contradicts \eqref{eq-claim-L-1-norm=of-u-at-initial-is-bounded-by-H-sub-m} and  the case follows.\\
	\textbf{Case 2.} {\it General case} $\left|\bold{u}\right|\left(\cdot,0\right)$ has compact support.\\ 
	Let $R_0>1$ be a constant such that
	\begin{equation*}
		\left|\bold{u}\right|(x,0)=0 \qquad \forall x\in\R^n\bs B_{R_0}
	\end{equation*}
	and consider the rescaled functions
	\begin{equation*}
		w^l(x,t)=\frac{1}{R_0^{\frac{p}{p-2}}}u^l\left(R_0x,t\right) \qquad \forall 1\leq i\leq k.
	\end{equation*}
	Then $\bold{w}=\left(w^1,\cdots,w^k\right)$ is a solution of \eqref{eq-system-PME-for-waiting-time-1} with 
	\begin{equation*}
		\textbf{supp}\left|\bold{w}\right|(\cdot,0)\subset B_1. 
	\end{equation*}
	By the \textbf{Case 1}, for any $0<R<R_0$ we can get
	\begin{equation}\label{eq-harnack-type-estimates-applying-case-1-for-general-case}
		\begin{aligned}
			\frac{1}{R_0^{n+\frac{p}{p-2}}}\int_{\left\{|x|<R\right\}}u^l(x,0)\,dx&=\int_{\left\{|x|<\frac{R}{R_0}\right\}}w^l(x,0)\,dx \\
			&\leq \frac{C}{\left(\mu_w^l\right)^{1+\frac{n(p-2)}{p}}}\left(\frac{1}{T^{\frac{1}{p-2}}}\frac{R^{n+\frac{p}{p-2}}}{R_0^{n+\frac{p}{p-2}}}+T^{\frac{n}{p}}\left(w^l\right)^{1+\frac{n(p-2)}{p}}\left(0,T\right)\right)\\
			&= \frac{C}{\left(\mu^l\right)^{1+\frac{n(p-2)}{p}}}\left(\frac{1}{T^{\frac{1}{p-2}}}\frac{R^{n+\frac{p}{p-2}}}{R_0^{n+\frac{p}{p-2}}}+\frac{T^{\frac{n}{p}}}{R_0^{n+\frac{p}{p-2}}}\left(u^l\right)^{1+\frac{n(p-2)}{p}}\left(0,T\right)\right) \qquad \forall 1\leq i\leq k
		\end{aligned}
	\end{equation}
	where
	\begin{equation*}
		\mu_w^l=\frac{\int_{\R^n}w^l(x,0)\,dx}{\max_{1\leq l\leq k}\left\{\int_{\R^n}w^l(x,0)\,dx\right\}} \qquad \forall 1\leq i\leq k.
	\end{equation*}
	Multiplying \eqref{eq-harnack-type-estimates-applying-case-1-for-general-case} by $R_0^{n+\frac{p}{p-2}}$, \eqref{eq-claim-L-1-norm=of-u-at-initial-is-bounded-by-H-sub-m} holds and the theorem follows.
\end{proof}

\indent As mentioned in Introduction, the result of Harnack type inequality plays an important role on the study of initial trace theorem. We finish this paper by providing the proof of Corollary \ref{cor-existence-of-initial-trace-from-Harnack-Type-Ineqlaity}. 

\begin{proof}[\textbf{Proof of Corollary \ref{cor-existence-of-initial-trace-from-Harnack-Type-Ineqlaity}}]
	By the law of $L^1$ mass conservation (Lemma \ref{lem-Law-of-L-1-Mass-Conservation}), $u^i\left(\cdot,t\right)$ is uniformly bounded in $L^1\left(\R^n\right)$. Thus, by an argument similar in the proof of Theorem \ref{thm-Main-boundedness-of-diffusion-coefficients-of-p-Laplacian}, the convergence \eqref{result-1-of-existence-of-initila-trace} can be easily proved.\\
	\indent We now focus on the dacay rate of initial trace $\rho^l$. By the result of Theorem \ref{lem-calculation-about-waiting-time-lower-bound}, there exists a constant $C_1=C_1(n,m)>0$ independent of $\epsilon\in\left(0,\frac{T}{2}\right)$ such that
	\begin{equation}\label{inequaity-applying-result-of-Thm-1-3-for=initial-trace}
		\begin{aligned}
			\int_{\left\{|x|<R\right\}}u^l(x,\epsilon)\,dx&\leq \frac{2^{\frac{1}{p-2}}\,C_1}{\left(\mu_0\right)^{1+\frac{n(p-2)}{p}}}\left(\frac{R^{n+\frac{p}{p-2}}}{T^{\frac{1}{p-2}}}+T^{\frac{n}{p}}\left(u^l\right)^{1+\frac{n(p-2)}{p}}\left(0,T\right)\right) \qquad\qquad  \forall R>0,\,\,1\leq l\leq k.
		\end{aligned}
	\end{equation}
	By \eqref{result-1-of-existence-of-initila-trace}, we can get
	\begin{equation}\label{local-convergence-between-u0-i-and-intial-trace}
		\lim_{\epsilon\to 0}\int_{\left\{|x|<R\right\}}u^l(x,\epsilon)\,dx=\int_{\left\{|x|<R\right\}}\rho^l(dx).
	\end{equation}
Since the constant $\mu_0$ is independent of $\epsilon$, by \eqref{inequaity-applying-result-of-Thm-1-3-for=initial-trace} and \eqref{local-convergence-between-u0-i-and-intial-trace}, \eqref{result-2-of-existence-of-initila-trace} hold for $C=\frac{2^{\frac{1}{p-2}}\,C_1}{\left(\mu_0\right)^{1+\frac{n(p-2)}{p}}}$ and the corollary follows.
\end{proof}

\section*{Acknowledgement} 
\setcounter{equation}{0}
\setcounter{thm}{0}

Ki-Ahm Lee was supported by Samsung Science and Technology Foundation under Project Number SSTF-BA1701-03. Ki-Ahm Lee also holds a joint appointment with Research Institute of Mathematics of Seoul National University. Sunghoon Kim was supported by the National Research Foundation of Korea(NRF) grant funded by the Korea government(MSIT) (No.  2020R1F1A1A01048334). Sunghoon Kim was also supported by the Research Fund, 2021 of The Catholic University of Korea.


\begin{thebibliography}{99}

\bibitem[A]{A} Martial Agueh, {\it Rates of decay to equilibria for $p$-Laplacian type equations.} Nonlinear Anal. 68 (2008), no. 7, 1909–1927.

\bibitem[AC]{AC} D. G. Aronson, L. A. Caffarelli, {\it The initial trace of a solution of the porous medium equation.} Trans. Amer. Math. Soc. 280 (1983), no. 1, 351-366.


\bibitem[CS]{CS} L.A. Caffarelli, S. Salsa, {\it A geometric approach to free boundary problems,} Graduate Studies in Mathematics, 68, American Mathematical Society, Providence, RI, 2005.

\bibitem[CV]{CV} L. Caffarelli, Alexis Vasseur, {\it Drift diffusion equations with fractional diffusion and the quasi-geometrophic
equation}  Annals of math. 171 (2010), 1903-1930.

\bibitem[DD]{DD} Manuel Del Pino, Jean Dolbeault, {\it Asymptotic behavior of nonlinear diffusions.} Math. Res. Lett. 10 (2003), no. 4, 551–557.

\bibitem[D]{D} E. DiBenedetto, {\it Degenerate Parabolic Equations,} Univertext, Springer-Verlag, New York, ISBN: 0-387-94020-0, 1993,
p. xvi+387.

\bibitem[DGV]{DGV} E. DiBenedetto, U. Gianazza, V. Vespri, {\it Harnack's inequality for degenerate and singular parabolic equations.} Springer Monographs in Mathematics. Springer, New York, 2012. xiv+278 pp.

\bibitem[DK]{DK} P. Daskalopoulos, C. E. Kenig, {\it Degenerate diffusions.
Initial value problems and local regularity theory.} EMS Tracts in Mathematics, 1. European Mathematical Society (EMS), Z\"urich, 2007. x+198 pp.

\bibitem[DUV]{DUV} E. Dibenedetto, J. M. Urbano, and V. Vespri, {\it Current issues on singular and degenerate evolution equations}, Handbook of Differential Equations, Evolutionary Equations, Vol. 1 (C. Dafermos and E. Feireisl, eds.), Elsevier, Amsterdam, 2004, pp. 169-286. MR 2103698(2006b:35160)

\bibitem[EG]{EG} L. C. Evans and R. F. Gariepy, {\it Wiener’s Criterion for the Heat Equation}, Arch. Rational Mech. Anal., 78(4), (1982), 293–314.

\bibitem[I]{I} T. Iwaniec, {\it Projections onto gradient fields and $L^p$-estimates for degenerated elliptic operators.} Studia Math. 75 (1983), no. 3, 293–312.
	
\bibitem[K]{K} A. S. Kalashnikov, {\it The heat equation in a medium with nonuniformly distributed nonlinear heat sources or absorbers.} (Russian) Vestnik Moskov. Univ. Ser. I Mat. Mekh. 1983, no. 3, 20–24.  

\bibitem[KL1]{KL1} Sunghoon Kim, Ki-Ahm Lee {\it Local Continuity and Asymptotic Behaviour of Degenerate Parabolic Systems}. Nonlinear Anal. 192 (2020), 111702, 32 pp. 

\bibitem[KL2]{KL2} Sunghoon Kim, Ki-Ahm Lee, {\it System of Porous Medium Equations},  J. Differential Equations 272 (2021), 433–472.

\bibitem[KV]{KV} S. Kamin, J.L. V\'azquez, {\it Fundamental solutions and asymptotic behaviour for the p-Laplacian equation,} Rev. Mat.
Iberoamericana 4 (2) (1988) 339–354.


\bibitem[La] {La} O. A. Ladyzhenskaya, {\it New equations for the description of the viscous incompressible fluids and global solvability in the range of the boundary value problems to these equations,} Trudy Steklov’s Math. Inst. 102 (1967) 85-104.

\bibitem[LPV]{LPV} Ki-ahm Lee, A. Petrosyan, J. L. V\'azquez, {\it Large-time geometric properties of solutions of the evolution $p$-Laplacian equation.} J. Differential Equations 229 (2006), no. 2, 389–411.

\bibitem[LSU]{LSU} O.A.~Ladyzenskaya, V.A.~Solonnikov, N.N.~Uraltceva, {\it Linear and quasilinear equations of parabolic type}, Transl. Math. Mono. vol. 23, Amer. Math. Soc., Providence, R.I., U.S.A., 1968.


\bibitem[M]{M} M. Misawa, {\it A H\"older estimate for nonlinear parabolic systems of p-Laplacian type}. J. Differential Equations 254 (2013), no. 2, 847–878.



\bibitem[U]{U} K. Uhlenbeck, {\it Regularity for a class of non-linear elliptic systems.} Acta Math. 138 (1977), no. 3-4, 219–240.

\bibitem[V]{V} J. L.  V\'azquez, {\it The porous medium equation. Mathematical theory.} Oxford Mathematical Monographs. The Clarendon Press, Oxford University Press, Oxford, 2007. xxii+624 pp.

\end{thebibliography}
\end{document}